\documentclass[a4paper]{article}
\usepackage{comment}
\usepackage{arxiv}
\usepackage[pagewise]{lineno}
\usepackage{bookmark}
\usepackage{multirow,multicol,diagbox,makecell}
\usepackage{booktabs}
\usepackage{graphicx}
\usepackage{algorithm}
\usepackage{algorithmic}
\usepackage{setspace}
\usepackage{amsmath,amsthm}
\usepackage{mathtools,esint,amssymb}
\usepackage{amsfonts}
\usepackage{enumerate}
\usepackage[numbers,sort]{natbib}
\usepackage{url}
\usepackage{xcolor}
\usepackage{tikz}
\usepackage{siunitx}
\usepackage{subcaption}
\usepackage{hyperref}
\usepackage[capitalise,nameinlink,noabbrev]{cleveref}
\usepackage{float}
\usepackage{placeins}

\usetikzlibrary{graphs, quotes, positioning, shapes.geometric, arrows.meta, fit, positioning, shapes.geometric, shapes.arrows}

\numberwithin{equation}{section}

\newcommand{\R}{\mathbb{R}}

\let\div\relax
\DeclareMathOperator{\div}{div} 
\DeclareMathOperator{\sgn}{sgn} 
\DeclareMathOperator{\Int}{Int} 

\newtheorem{theorem}{Theorem}
\newtheorem{remark}{Remark}
\newtheorem{lemma}{Lemma}
\newtheorem{proposition}{Proposition}
\newtheorem{definition}{Definition}

\title{Recovering the spatiotemporally dependent diffusion and advection coefficients in a pathway-based Keller--Segel model}

\renewcommand{\shorttitle}{Recovering the diffusion and advection coefficients in the Keller--Segel model}

\author{Yu'an Li$^{1,2}$, Lingyun Qiu$^{3,4}$, Min Tang$^{1,5}$, Hui Yu$^{6}$, Shenwen Yu$^{7}$, Siqin Zheng$^{8}$\\
$^1$ Institute of Natural Sciences, Shanghai Jiaotong University, Shanghai 200240, China\\
$^2$ School of Physics and Astronomy, Shanghai Jiaotong University, Shanghai 200240, China\\
(\href{mailto:yuan_li@sjtu.edu.cn}{yuan\_li@sjtu.edu.cn})\\
$^3$ Yau Mathematical Sciences Center, Tsinghua University, Beijing 100084, China\\
$^4$ Yanqi Lake Beijing Institute of Mathematical Sciences and Applications, Beijing 101408, China\\
(\href{mailto:lyqiu@tsinghua.edu.cn}{lyqiu@tsinghua.edu.cn})\\
$^5$ School of Mathematics, Shanghai Jiaotong University, Shanghai 200240, China\\
(\href{mailto:tangmin@sjtu.edu.cn}{tangmin@sjtu.edu.cn})\\
$^6$ School of Mathematics and Computational Science, Xiangtan University, Xiangtan 411105, Hunan, China\\
(\href{mailto:huiyu@xtu.edu.cn}{huiyu@xtu.edu.cn})\\
$^7$ Department of Mathematical Sciences, Tsinghua University, Beijing 100084, China\\
(\href{mailto:ysw22@mails.tsinghua.edu.cn}{ysw22@mails.tsinghua.edu.cn})\\
$^8$ Department of Mathematics, City University of Hong Kong, Kowloon, Hong Kong SAR, China\\
(\href{mailto:siqzheng@cityu.edu.hk}{siqzheng@cityu.edu.hk})
}

\begin{document}
\maketitle

\begin{abstract}
We investigate an inverse coefficient problem for an augmented Keller--Segel model arising in the description of phototactic and chemotactic population dynamics. The spatiotemporal evolution of the population density is governed by an advection-diffusion equation in which both the diffusion coefficient and the advection coefficient depend on space and time. Due to the underlying intracellular adaptation mechanism, even in a time-independent external environment, the diffusion coefficient \(D(x,t)\) and the scalar advection coefficient \(K(x,t)\) remain time-dependent and relax exponentially toward their respective steady-state profiles. Importantly, their relaxation is governed by the same rate, determined by the intracellular adaptation dynamics.
From internal measurements of the density, we establish conditional Lipschitz stability for recovering either coefficient when the other is known, and derive cross-sensitivity estimates quantifying the compensation between diffusion and advection perturbations that produce the same density data.
We further identify a finite-time separation between the transient and steady regimes: the full dynamics admit an exponentially accurate frozen-coefficient approximation at late times.
When the steady-state coefficients are known, we also prove unique identification of the relaxation rate from the long-time behavior of the solution.
Finally, we derive an explicit linearization of the forward map and, motivated by the temporal separation above, develop a two-stage gradient-based reconstruction strategy.
Numerical experiments demonstrate the feasibility of the proposed method.
\end{abstract}



\section{Introduction}\label{sec:intro}


The Keller--Segel model provides a classical framework for describing the collective dynamics of microbial populations by integrating two key components: random diffusive motion and directed advective transport driven by external chemical or light gradients. This model has been widely adopted to characterize the spatiotemporal evolution of cell density distributions. 
In the present study, we focus on a pathway-based Keller--Segel model defined over a bounded spatial domain $\Omega\subset\mathbb{R}^2$. We write $\rho(x,t)$ for the microbial population density, where $x\in\Omega$ denotes the spatial position and $t>0$ denotes the time. For a prescribed external signal $I(x,t)$, the dynamics of the population density are governed by the following initial-boundary value problem: 
\begin{equation}\label{eq:model}
    \begin{cases}
        \partial_t \rho - \operatorname{div}(D(x,t)\nabla\rho) + \operatorname{div}(\rho \Gamma(x,t)) = 0 & \text{in}\,\ Q\coloneqq\Omega\times(0,\infty), \\
        \partial_\nu \rho = 0 & \text{on}\ \Sigma\coloneqq\partial\Omega\times(0,\infty), \\
        \rho(\cdot,0) = \varphi & \text{in}\ \ \Omega,
    \end{cases}
\end{equation}
where $D(x,t)$ and $\Gamma(x,t)$ represent the cell diffusion coefficient and drift vector, respectively, describing the cellular response to the given external signal $I(x,t)$, and $\nu$ denotes the unit outward normal vector to the domain boundary $\partial\Omega$. 
The corresponding total cell flux is $J = D\nabla\rho - \rho\Gamma$. We assume that the illumination is high in the interior of the container and induces no normal phototactic drift on the boundary, so that
\begin{equation}\label{eq:drift_boundary}
    \Gamma\cdot\nu = 0 \ \text{on}\ \Sigma.
\end{equation}
Under this assumption, the no-flux boundary condition $J\cdot\nu = 0$, which reflects the confinement of the algae within the closed container, reduces to the homogeneous Neumann condition $\partial_\nu \rho = 0$ used above.

Microbial cells detect and respond to environmental cues through complex intracellular signaling networks, which implies that the effective diffusion and advection coefficients depend on external signals in a highly nonlinear and indirect manner. 
Consequently, determining the dependence of $D(x,t)$ and $\Gamma(x,t)$ on $I(x,t)$ remains a nontrivial task. One strategy is a bottom-up approach, in which the dependence of both coefficients on external signals is derived from mechanistic descriptions of the underlying chemotactic or phototactic signaling pathways \cite{Li2016,Perthame2020}. Nevertheless, this approach requires detailed knowledge of intracellular molecular regulatory mechanisms, which has only been established for intensively studied model organisms such as \textit{E. coli} and is largely unavailable for most other microbial species. An alternative strategy is to infer these coefficients from population-level density measurements by solving inverse problems, thereby reconstructing the relationships between the coefficients and external signals. A major challenge in this endeavor arises from cellular adaptation processes: even when the external signal $I$ varies only in space, the effective diffusion and advection coefficients may still depend on both space and time, leading to highly dynamic and environment-dependent transport properties.

In this work, we investigate an inverse problem motivated by phototaxis experiments on algae, where the population is subjected to controlled light illumination within a container, represented by the bounded domain $\Omega\subset\R^2$.
Given spatiotemporal measurements of the algal population density 
\(\rho(x,t)\) within $\Omega\times(0,T)$, we aim to recover the model parameters
that characterize the phototactic response: the diffusion coefficient
\(D(x,t)\), reflecting the intensity of random motion,
and the advection vector \(\Gamma(x,t)\), encoding the intensity
and gradient of the outside signal.

The simultaneous recovery of $D(x,t)$ and $\Gamma(x,t)$ from a single density evolution is highly underdetermined in a completely general setting, since the scalar observation $\rho(x,t)$ is governed by an equation in which the scalar diffusion coefficient and the vector-valued drift are strongly coupled. We therefore restrict attention to a structured class motivated by the pathway-based Keller--Segel model and the long-time behavior of the coefficients. This framework is physically meaningful and experimentally tractable, while retaining parameters with clear interpretations in terms of the cellular response to illumination.

More precisely, the pathway-based Keller--Segel model and the adaptation mechanism lead, after an initial transient, to a reduced representation of the transport coefficients. For notational convenience, we take the beginning of this relaxation regime as $t=0$. The diffusion coefficient is represented as
\begin{equation*}
    D(x,t) = D_s(x) + e^{-\gamma t}[D_0(x) - D_s(x)],
\end{equation*}
while the drift field takes the form $\Gamma(x,t) = K(x,t)\nabla\Phi(x)$, with a prescribed illumination profile $\Phi$ and a scalar advection coefficient
\begin{equation*}
    K(x,t) = K_s(x) + e^{-\gamma t}[K_0(x) - K_s(x)].
\end{equation*}
Here $D_s$ and $K_s$ denote the steady-state coefficients, $D_0$ and $K_0$ denote the coefficients at the beginning of the reduced regime, and $\gamma>0$ is the relaxation rate. In this setting, the inverse problem reduces to the simultaneous recovery of the five parameters $D_0, D_s, K_0, K_s, \gamma$. This reduced structure retains the essential adaptation dynamics while making the inverse problem mathematically tractable.

\subsection{Literature review}

Reconstructing coefficients from solutions of governing partial
differential equations is a classical inverse problem,
with stability being a central issue:
it quantifies how measurement errors propagate into uncertainties
in the recovered parameters and thus determines the reliability
of the inversion.
For parabolic equations, stability has been extensively studied
using Carleman estimates
\cite{Imanuvilov2022,Imanuvilov2024,Imanuvilov2024a,Lin2022b}
and boundary control techniques \cite{Qiu2025}.
These works mainly concern reconstructions from boundary measurements
or from data supplemented by a single time snapshot.

The present problem instead involves internal measurements of the complete density evolution \(\rho(x,t)\) throughout the entire spatial domain \(\Omega\) over a time interval \((0,T)\).
Inverse problems with internal data typically exhibit enhanced stability
compared with those relying solely on boundary observations,
as shown in
\cite{Kuchment2012,Bal2013a,Bonito2017,Choulli2019a,Choulli2021,Chen2026a}.
A common requirement in such analysis is a nondegeneracy condition
on the solution or its gradient.
In the present scenario, the intrinsic properties of the model,
including the nonnegativity and the maximum principle
\cite{Lieberman1996,Perthame2015},
ensure that the measured density is strictly positive. 
While this property is beneficial for recovering the advection coefficient, 
the reconstruction of the diffusion coefficient remains challenging, 
because conditions involving $\nabla\rho$ naturally arise.

Even in the simpler case of a spatially dependent diffusion coefficient \(D(x)\), identification from internal measurements at only one or several fixed times remains challenging; see
\cite{Alessandrini2020}.
In one dimension, uniqueness and conditional stability
have been obtained via a heat equation transform
\cite{isakovIdentificationDiffusionCoefficient2000},
and Carleman-based conditional stability estimates
under additional gradient conditions were derived in
\cite{Yamamoto2009}.
For numerical recovery, error analysis using
Galerkin finite element methods
was developed in \cite{jinErrorAnalysisFinite2021}.
These studies concern either spatially dependent coefficients or observation settings different from ours. 

The simultaneous recovery of multiple coefficients, however, significantly increases the difficulty of the inverse problem.
The coupling between parameters may lead to severe ill-posedness, and stability may deteriorate substantially.
In certain configurations, the parameters can be decoupled
and reconstructed sequentially \cite{Cen2024,Pan2025}.
While for the general model, to the best of our knowledge, the simultaneous identification of a spacetime-dependent diffusion coefficient $D(x,t)$ and an advection coefficient $K(x,t)$ from full spatiotemporal internal measurements has not been theoretically investigated.

\subsection{Main results}

Our theoretical analysis consists of four main parts.
First, we investigate the conditional recovery of the diffusion coefficient $D$ and the scalar advection coefficient $K$ from the full internal observation of the population density. When $D$ is known, we establish a Lipschitz stability estimate for the recovery of $K$. Conversely, when $K$ is prescribed, we obtain a corresponding Lipschitz stability estimate for $D$ under an appropriate nondegeneracy condition on the spatial gradient of the measured density. The two reconstruction problems exhibit different analytical features: the recovery of $K$ benefits from the strict positivity of $\rho$, whereas the identification of $D$ necessarily involves $\nabla\rho$ and therefore requires additional control of its degeneracy.

Second, we analyze the cross-sensitivity between $D$ and $K$ when two coefficient pairs generate the same internal density. The difference of the two equations yields a compensation identity, from which we derive estimates controlling a perturbation of either coefficient in terms of the other. These estimates quantify the intrinsic diffusion--advection coupling in the simultaneous inverse problem and may also be viewed from the perspective of uncertainty quantification \cite{Ren2019,Ren2020}.

Third, we exploit the exponential coefficient structure to quantify the separation between the transient and asymptotic regimes in finite observation intervals. In a sufficiently late-time interval $[T_2,T]$, replacing the time-dependent coefficients by their steady states produces a forward-model error of order $e^{-\gamma T_2}$. In the complementary early-time interval $[0,T_1]$, perturbations of the steady coefficients affect the solution with the attenuation factor $1-e^{-\gamma T_1}$. These estimates directly show quantitatively that late-time observations are predominantly governed by the steady coefficients, while early-time dynamics are comparatively insensitive to moderate perturbations of them.

Fourth, we address the recovery of the common relaxation rate $\gamma$. Assuming that the steady-state coefficients $D_s$ and $K_s$ are known, we establish exponential convergence of the density to its steady state and show that $\gamma$ can be uniquely identified from the long-time decay of an observable residual relative to the steady-state operator.

For numerical reconstruction, we further derive the Fr\'echet derivatives of the forward map and the associated adjoint state equations; see \cite{plessixReviewAdjointstateMethod2006}.
The adjoint-state framework is widely used in inverse problems
\cite{isakovIdentificationDiffusionCoefficient2000,
riederMathematicalFoundationFull2025,
ammariPhasedPhaselessDomain2016a}.
Under H\"older-type stability assumptions,
error estimates of gradient-based methods for solving inverse problems
can be derived as in
\cite{dehoopLocalAnalysisInverse2012,
dehoopAnalysisMultilevelProjected2015}.

The theoretical results above suggest a natural temporal decomposition of the simultaneous reconstruction problem. We therefore develop a two-stage reconstruction strategy that exploits different observation windows. In Stage 1, the steady coefficients $D_s$ and $K_s$ are reconstructed from late-time data using a steady-coefficient model. In Stage 2, the reconstructed steady coefficients are fixed and the remaining transient parameters, including the initial coefficient profiles $D_0, K_0$ and the relaxation rate $\gamma$, are recovered from early-time observations. The finite-time estimates above quantify the temporal separation underlying this strategy, while the numerical experiments demonstrate its practical feasibility.

The remainder of this paper is organized as follows.
In \cref{sec:eq}, we introduce the pathway-based Keller--Segel model and derive the reduced structure of the diffusion and advection coefficients by analyzing their long-time behavior. 
We also recall basic results for the forward problem, including well-posedness and qualitative properties of solution.
In \cref{sec:ip}, we establish conditional stability and cross-sensitivity estimates for the diffusion and advection coefficients, quantify the finite-time separation between the transient and asymptotic regimes, and analyze the long-time identification of the relaxation rate.
We also derive the Fr\'echet derivatives of the forward map and their adjoint operators.
In \cref{sec:num}, motivated by these analytical results, we develop a two-stage reconstruction strategy and present the corresponding numerical results.
Finally, concluding remarks are given in \cref{sec:con}.

\section{The pathway-based Keller--Segel model}\label{sec:eq}

The classical continuum framework for describing population-level cell migration in heterogeneous environments is the Keller--Segel model. In its standard form, when the external signal \(I(x,t)\) is independent of time, the effective diffusion coefficient \(D\) and drift vector $\Gamma$ are usually assumed to depend only on space, so that \(D=D(x)\) and \(\Gamma=\Gamma(x)\). This assumption implicitly requires that the intracellular signaling and adaptation dynamics are much faster than the population-level motion. In other words, each cell is assumed to be fully adapted, with its intracellular chemical state instantaneously equilibrated to the local value of the external signal.

This quasi-steady assumption is not valid in many experimentally relevant situations. A number of experiments on both \textit{E. coli} and \textit{Euglena} exhibit delayed, or adaptation-dependent migration patterns that cannot be reproduced by the standard Keller--Segel model \cite{Zhu2012,Zhang2019}. Even when the external signal is stationary, cells may initially possess intracellular states that are not at the equilibrium associated with the local signal. Moreover, when \(I(x,t)\) changes on a timescale comparable to the intracellular adaptation time, the cellular response depends not only on the instantaneous signal but also on the history.

Since the measured dynamics retain information about the cells’ initial intracellular states and their previous exposure to the signal, the discrepancy cannot be removed simply by adjusting $D(x)$ and $\Gamma(x)$; it reflects a structural deficiency of the standard Keller--Segel closure, which eliminates the intracellular adaptation timescale altogether. To fit with the experimental data quantitatively, one must introduce an extended Keller--Segel model in which the effects of intracellular adaptation are incorporated through additional dynamical variables.

In \cite{Li2016,Lei}, by taking into account the effects of intracellular adaptation, a pathway-based Keller--Segel model is derived from a bottom-up approach starting from the individual-based model.
The collective movement of microbial populations is governed by the following equation:
\begin{equation}\label{eq:adv-diff-M}
    \begin{cases}
        \partial_t \rho - \operatorname{div}\bigl(\hat D(M)\nabla\rho\bigr) + \operatorname{div}\bigl(\rho \hat K(M)\nabla M\bigr) = 0 & \text{in }\, Q, \\
        \partial_t M(x,t) = \gamma\bigl(\hat\Phi(I)-M\bigr) & \text{in }\, Q, \\
       \partial_\nu \rho = 0 & \text{on } \Sigma, \\
       \rho(\cdot,0) = \varphi & \text{in }\ \Omega.
    \end{cases}
\end{equation}
Here we have changed the spatial-temporal dependent diffusion and advection coefficients $D(x,t)$ and $K(x,t)$ to ones with functional dependence of a new variable $M(x,t)$. It denotes the average methylation level for \textit{E. coli} (or the level of other internal proteins for other microbial species) and represents the slowest intracellular reaction within the chemotaxis/phototaxis signaling pathway. The diffusion and advection coefficients depend on $M$ in a nonlinear way that rely on the detailed intracellular signaling pathway.
The function \(\hat\Phi(I)\) is a complex nonlinear mapping of the external signal \(I\), which defines the equilibrium protein level for a given stimulus: \(I\) may correspond to light intensity for algal phototaxis or chemical concentration for \textit{E. coli} chemotaxis. The relaxation rate $\gamma$ is inherited from the single-cell relaxation rate rather than being an independent macroscopic parameter. It is obtained by the time scale over which individual cells adapt their intracellular signaling state, and by nondimensionalizations \cite{Li2016,Lei}.

Nevertheless, the explicit functional form of $\hat D(M)$, $\hat K(M)\nabla M$, \(\hat\Phi(I)\), the parameter \(\gamma\), and direct measurements of \(M\) are typically unavailable for most microbial species. Consequently, it is necessary to reduce the model's dependence on the internal variable \(M\).

\subsection{The reduced model}

Before addressing the more general and mathematically demanding case of a spatiotemporally varying light field \(I(x,t)\), it is useful to first consider a static but spatially heterogeneous signal \(I(x)\). This setting is not merely a convenient simplification. It already provides a physically meaningful framework in which the consequences of intracellular adaptation can be isolated from those caused by the explicit temporal variation of the external environment. In particular, even though the imposed light field remains fixed after its introduction, the cell population may exhibit strongly time-dependent transport behavior because the intracellular states of the cells require a finite time to adapt to the new local signal.

Denoting $\Phi(x) = \hat\Phi(I(x))$, one can solve the second equation in \eqref{eq:adv-diff-M} to derive
\begin{equation}\label{eq:M}
    M(x,t) = \Phi(x) + e^{-\gamma t}[M(x,0) - \Phi(x)].
\end{equation}
Note that as $t\to\infty$, $M$ tends to $\Phi$, which represents the steady state under the given external signal $I$. Then the diffusion and advection coefficients $\hat D(M), \hat K(M)\nabla M$ tend to $\hat D(\Phi), \hat K(\Phi)\nabla\Phi$ respectively. To simplify the following analysis and numerical reconstruction, we first investigate the long-time behavior of these coefficients.

\begin{proposition}
    Suppose that $\hat D, \hat K$ are $C^2$ functionals and $\Phi\in C^1(\overline{\Omega})$. Then for $t>t_0\gg 1$ it holds that
    \begin{align*}
        \hat{D}(M) &= \hat{D}(\Phi) + e^{-\gamma(t-t_0)}[\hat{D}(M(x,t_0)) - \hat{D}(\Phi)] + O(e^{-\gamma(t+t_0)}), \\
        \hat{K}(M)\nabla M &= \big[\hat{K}(\Phi) + e^{-\gamma(t-t_0)}\big(\hat{K}(M(x,t_0)) - \hat{K}(\Phi)\big)\big]\nabla\Phi + O(e^{-\gamma t}).
    \end{align*}
\end{proposition}
\begin{proof}
    By Taylor's expansion it follows from \eqref{eq:M} that
    \begin{equation*}
        \hat{D}(M(x,t))
        = \hat{D}(\Phi) + e^{-\gamma t}\hat{D}'(\Phi)[M(x,0) - \Phi(x)] + O(e^{-2\gamma t}) \ \text{as}\ t\to\infty.
    \end{equation*}
    Taking a time $t_0$ such that $e^{-\gamma t_0} \ll 1$, we derive
    \begin{equation*}
        \hat{D}(M(x,t_0)) = \hat{D}(\Phi) + e^{-\gamma t_0}\hat{D}'(\Phi)[M(x,0) - \Phi(x)] + O(e^{-2\gamma t_0}),
    \end{equation*}
    so that
    \begin{equation*}
        \hat{D}'(\Phi)[M(x,0) - \Phi(x)] = e^{\gamma t_0}[\hat{D}(M(x,t_0)) - \hat{D}(\Phi)] + O(e^{-\gamma t_0}).
    \end{equation*}
    Thus, for $t>t_0$,
    \begin{equation*}
        \hat{D}(M(x,t)) = \hat{D}(\Phi) + e^{-\gamma(t-t_0)}[\hat{D}(M(x,t_0)) - \hat{D}(\Phi)] + O(e^{-\gamma(t+t_0)}).
    \end{equation*}

    For the advection vector, it follows in the same way that for $t>t_0$,
    \begin{align*}
        \hat{K}(M)\nabla M ={}& \big[\hat{K}(\Phi) + e^{-\gamma(t-t_0)}\big(\hat{K}(M(x,t_0)) - \hat{K}(\Phi)\big) + O(e^{-\gamma(t+t_0)})\big]\big[\nabla\Phi + e^{-\gamma t}\big(\nabla M(x,0)-\nabla\Phi\big)\big] \\
        ={}& \big[\hat{K}(\Phi) + e^{-\gamma(t-t_0)}\big(\hat{K}(M(x,t_0)) - \hat{K}(\Phi)\big)\big]\nabla\Phi + O(e^{-\gamma t}). \qedhere
    \end{align*}
\end{proof}

According to the long-time behavior, we fix a sufficiently large time $t_0$ and approximate the coefficients $\hat{D}(M(x,t))$ and $\hat{K}(M)\nabla M$ by $D(x,t)$ and $K(x,t)\nabla\Phi$ respectively by dropping the higher-order terms. The reduced coefficients are written as
\begin{equation}\label{eq:coeff}
    \begin{aligned}
        D(x,t) &= D_s(x) + e^{-\gamma(t-t_0)}[D_0(x) - D_s(x)], \\
        K(x,t) &= K_s(x) + e^{-\gamma(t-t_0)}[K_0(x) - K_s(x)]
    \end{aligned}
\end{equation}
for $t>t_0$, where
\begin{align*}
    D_s(x) &= \hat{D}(\Phi(x)), \quad D_0(x) = \hat{D}(M(x,t_0)), \\
    K_s(x) &= \hat{K}(\Phi(x)), \quad K_0(x) = \hat{K}(M(x,t_0)).
\end{align*}
We conclude this subsection with the reduced Keller--Segel model:
\begin{equation}\label{eq:adv-diff}
    \begin{cases}
        \partial_t \rho - \operatorname{div}(D(x,t)\nabla\rho) + \operatorname{div}(\rho K(x,t)\nabla\Phi(x)) = 0 & \text{in}\,\ Q, \\
        \partial_\nu \rho = 0 & \text{on}\ \Sigma, \\
        \rho(\cdot,0) = \varphi & \text{in}\ \ \Omega.
    \end{cases}
\end{equation}
For consistency with the assumption \eqref{eq:drift_boundary}, we assume that $\Phi$ satisfies the homogeneous Neumann boundary condition $\partial_\nu\Phi = 0$ on $\partial\Omega$.

\subsection{Preliminary results}

The rest of this section is devoted to presenting some useful results on the parabolic equation with Neumann boundary condition. In the following, we sometimes abbreviate a spatiotemporal function $u(x,t)$ by $u(t)$, which represents a spatial function for each time $t$, and denote its time-derivative by $u'(t)$. We use $C$ to denote generic positive constants in estimates.

First we define the weak solution and mention the well-posedness result for \eqref{eq:adv-diff} in a finite time interval $(0,T)$. Here we involve a source term, so that the results are useful for deriving the Fr\'echet derivatives of the forward map and their adjoint states in \cref{sec:ip}.
\begin{definition}
    We say that $\rho\in L^2(0,T;H^1(\Omega))$ with $\partial_t\rho\in L^2(0,T;H^1(\Omega)^*)$ is a weak solution of
    \begin{equation}\label{eq:with_source}
        \begin{cases}
            \partial_t \rho - \div(D\nabla\rho) + \div(\rho\Gamma) = \div f & \text{in}\,\ Q_T\coloneqq \Omega\times(0,T), \\
            \partial_\nu \rho = -f\cdot\nu & \text{on}\ \Sigma_T\coloneqq \partial\Omega\times(0,T), \\
            \rho(\cdot,0) = \varphi & \text{in}\,\ \Omega,
        \end{cases}
    \end{equation}
    if
    \begin{equation}\label{eq:weak}
        \langle\partial_t\rho, v\rangle + \int_\Omega (D\nabla\rho - \rho\Gamma)\cdot\nabla v\,dx = -\int_\Omega f\cdot\nabla v\,dx
    \end{equation}
    for any $v\in H^1(\Omega)$ and a.e. $t\in[0,T]$, and $\rho(0) = \varphi$.
\end{definition}

\begin{remark}
    By the embedding property in \cite[Theorem 7.100]{Salsa2022}, it follows automatically from the definition that $\rho\in C([0,T];L^2(\Omega))$. So the initial condition $\rho(0) = \varphi$ makes sense in $L^2(\Omega)$.
\end{remark}

\begin{remark}\label{rmk:Neumann}
    The weak formulation \eqref{eq:weak} with $f=0$ naturally incorporates the no-flux boundary condition for the total cell flux:
    \begin{equation*}
        D\partial_\nu \rho - \rho\Gamma\cdot\nu = 0 \ \text{on}\ \Sigma_T,
    \end{equation*}
    which is the boundary condition originally associated with the conservation law in \eqref{eq:model}. As discussed after the model formulation, we assume that the drift induced by light illumination has no normal component on the boundary, namely \eqref{eq:drift_boundary}, or equivalently $\partial_\nu\Phi = 0$ in the reduced model \eqref{eq:adv-diff}. Under this assumption, the no-flux condition reduces to the Neumann condition, so that the strong and weak formulations are consistent.
\end{remark}

The following results, including the well-posedness and some basic properties of the weak solution, are standard and well-known. For their proofs, one can refer to, e.g., \cite{Evans2010,Perthame2015,Lieberman1996}.

\begin{proposition}\label{prop:well_posed}
    Suppose $\Omega$ is a bounded domain in $\R^2$, and $D,\Gamma\in L^\infty(Q_T)$ satisfy
    \begin{equation}\label{eq:as}
        \Gamma\cdot\nu = 0\ \text{on}\ \Sigma_T,\quad
        D\ge\lambda\ \text{in}\ Q_T,\quad \text{and}\quad
        \|D\|_{L^\infty(Q_T)}, \|\Gamma\|_{L^\infty(Q_T)}\le\Lambda
    \end{equation}
    for some positive constants $\lambda, \Lambda$. Then for any $f\in L^2(Q_T)$ and $\varphi\in L^2(\Omega)$, the problem \eqref{eq:with_source} admits a unique weak solution $\rho(x,t)$, and it satisfies:
    \begin{enumerate}[(i)]
        \item (mass conservation)
        \begin{equation*}
            \int_\Omega \rho(x,t)\,dx \equiv \int_\Omega \varphi(x)\,dx, \quad \forall\,t\in[0,T];
        \end{equation*}

        \item (energy estimate)
        \begin{equation*}
            \|\rho\|_{L^2(0,T;H^1(\Omega))} \le C(\lambda,\Lambda,\Omega,T)\big(\|f\|_{L^2(Q_T)} + \|\varphi\|_{L^2(\Omega)}\big);
        \end{equation*}
        
        \item (nonnegativity principle) if $f = 0$ and $\varphi\ge 0$, then $\rho\ge 0$ in $Q_T$.
    \end{enumerate}
\end{proposition}

It is remarkable that the nonnegativity principle is consistent with the fact that the population density should be nonnegative. Moreover, by the strong maximum principle of the Neumann problem of parabolic equation \cite[Theorem 6.43]{Lieberman1996}, we can see that the solution cannot vanish anywhere in $\overline{\Omega}\times(0,T]$, thereby having a positive lower bound in the compact set $\overline{\Omega}\times[\tau, T]$ for any $\tau\in(0,T)$. This will be helpful in the stability analysis of the inverse problem. To illustrate it rigorously, we need a higher regularity result for the weak solution:
\begin{proposition}\label{prop:high_reg}
    Suppose $\Omega$ is a bounded domain in $\R^2$ with $C^2$ boundary, and $D,\Gamma\in C^1([0,T];W^{1,\infty}(\Omega))$ satisfying \eqref{eq:as} and
    \begin{equation}\label{eq:as2}
        \|D\|_{C^1([0,T];W^{1,\infty}(\Omega))}, \|\Gamma\|_{C^1([0,T];W^{1,\infty}(\Omega))} \le A
    \end{equation}
    for some constant $A>0$. Then for any $\varphi\in H^2(\Omega)$, the weak solution to \eqref{eq:with_source} with $f=0$ satisfies $\rho\in C([0,T];H^2(\Omega))$, and
    \begin{equation*}
        \|\rho\|_{C([0,T];H^2(\Omega))} \le C(\lambda,A,\Omega,T)\|\varphi\|_{H^2(\Omega)}.
    \end{equation*}
\end{proposition}
\begin{proof}
    Analogously to the proof of Theorem 5 in \cite[Section 7.1]{Evans2010} for Dirichlet problem, but with additional integrations by parts since our coefficients $D,\Gamma$ are time-dependent, we derive that $\rho\in L^\infty(0,T;H^2(\Omega))$ and $\partial_t\rho\in C([0,T];L^2(\Omega))$, with
    \begin{equation}\label{eq:uniform}
        \|\rho\|_{L^\infty(0,T;H^2(\Omega))} \le C(\lambda,A,\Omega,T)\|\varphi\|_{H^2(\Omega)}.
    \end{equation}
    So it remains to prove that $\rho\colon [0,T]\to H^2(\Omega)$ is continuous. For this purpose, consider the equation of the difference $u(t;s) = \rho(t) - \rho(s)$ for any $t,s\in[0,T]$:
    \begin{align*}
        -\div(D(t)\nabla u(t;s)) + \div(u(t;s)\Gamma(t)) = \rho'(s) - \rho'(t) &+ \div((D(t)-D(s))\nabla\rho(s)) \\
        &- \div(\rho(s)(\Gamma(t)-\Gamma(s))) \ \text{in}\ \Omega,
    \end{align*}
    with Neumann boundary condition $\partial_\nu u(t;s) = 0$ on $\partial\Omega$. By the regularity result of elliptic equation, since $\int_\Omega u(t;s)\,dx = 0$ according to the mass conservation, we know that
    \begin{align*}
        \|u(t;s)\|_{H^2(\Omega)} \le{}& C(\lambda,A,\Omega) \big\|\rho'(s)-\rho'(t) + \div((D(t)-D(s))\nabla\rho(s)) - \div(\rho(s)(\Gamma(t)-\Gamma(s)))\big\|_{L^2(\Omega)} \\
        \le{}& C(\lambda,A,\Omega) \big(\|\rho'(s)-\rho'(t)\|_{L^2(\Omega)} + \|D(t)-D(s)\|_{W^{1,\infty}(\Omega)}\|\rho(s)\|_{H^2(\Omega)} \\
        &\phantom{C(\lambda,A,\Omega)}\ + \|\Gamma(t)-\Gamma(s)\|_{W^{1,\infty}(\Omega)}\|\rho(s)\|_{H^1(\Omega)}\big) \\
        \le{}& C(\lambda,A,\Omega,T) \big(\|\rho'(s)-\rho'(t)\|_{L^2(\Omega)} + \|D(t)-D(s)\|_{W^{1,\infty}(\Omega)}\|\varphi\|_{H^2(\Omega)} \\
        &\phantom{C(\lambda,A,\Omega,T)}\ + \|\Gamma(t)-\Gamma(s)\|_{W^{1,\infty}(\Omega)}\|\varphi\|_{H^2(\Omega)}\big),
    \end{align*}
    where we used the estimate \eqref{eq:uniform} in the last inequality. Since $\partial_t\rho\in C([0,T];L^2(\Omega))$ and $D,\Gamma\in C([0,T];W^{1,\infty}(\Omega))$, we see that $u(t;s)\to 0$ in $H^2(\Omega)$ as $s\to t$ for any $t\in[0,T]$, meaning that $\rho\in C([0,T];H^2(\Omega))$.
\end{proof}

This proposition shows that $\rho\in C(\overline{Q_T})$ under the given conditions by Sobolev embedding. If the initial population $\varphi$ is nonnegative and nontrivial, the strong maximum principle \cite[Theorem 6.43]{Lieberman1996} implies that the solution is strictly positive for every positive time. Combined with the continuity obtained above, this shows that every fixed solution admits a positive lower bound on any compact time interval bounded away from the initial time.

\section{The inverse problem}\label{sec:ip}

By translation, one can shift $t_0$ to the initial time and still denote the final observation time by $T$.
The diffusion coefficient $D$ and the advection coefficient $K$ are therefore expressed as
\begin{equation}\label{eq:para}
    \begin{split}
        D(x,t) &= D_s(x) + e^{-\gamma t}[D_0(x) - D_s(x)], \\
        K(x,t) &= K_s(x) + e^{-\gamma t}[K_0(x) - K_s(x)].
    \end{split}
\end{equation}
Since the translated initial time corresponds to a positive time $t_0$ of the original evolution, the measured density is strictly positive on the translated observation cylinder. Hence, for each fixed observed data $\rho$,
\begin{equation*}
    c_\rho \coloneqq \min_{\overline{Q_T}} \rho > 0.
\end{equation*}
This value, however, is solution-dependent and is therefore retained explicitly in the assumptions of the stability estimates below.

Throughout this section, we assume that $\Omega$ is a bounded domain in $\R^2$ with $C^2$ boundary and $\nabla\Phi\in W^{1,\infty}(\Omega)$, and retain the standing boundary condition $\partial_\nu\Phi=0$ on $\partial\Omega$. 
The inverse problem seeks to recover the five parameters, $D_0(x),D_s(x),K_0(x),K_s(x)$ and $\gamma$, from the recorded data $\rho|_{Q_T}$.
To formulate the forward operator $F$ of this inverse problem, we introduce an auxiliary operator
\begin{equation*}
    \tilde{F}\colon {\cal D}(\tilde{F})\subset L^\infty(Q_T)^2 \to L^2(Q_T), \quad (D,K)\mapsto \rho|_{Q_T}
\end{equation*}
with domain
\begin{equation*}
    {\cal D}(\tilde{F}) = \big\{(D,K)\in L^\infty(Q_T)^2: D\ge\lambda\ \text{in}\ Q_T\ \text{and}\
    \|D\|_{L^\infty(Q_T)}, \|K\|_{L^\infty(Q_T)}\le\Lambda\big\}.
\end{equation*}
Write the parameter expression of \eqref{eq:para} as an operator
\begin{equation*}
    P\colon {\cal D}(P)\subset L^\infty(\Omega)^4\times\R \to L^\infty(Q_T)^2, \quad (D_0,D_s,K_0,K_s,\gamma)\mapsto (D,K)
\end{equation*}
with domain ${\cal D}(P) = {\cal T}_+^2 \times {\cal T}^2 \times \R_+$, where
\begin{equation*}
    {\cal T} = \{w\in L^\infty(\Omega): \|w\|_{L^\infty(\Omega)}\le\Lambda\}, \quad
    {\cal T}_+ = \{w\in {\cal T}: w\ge\lambda\ \text{in}\ \Omega\}
\end{equation*}
for some positive constants $\lambda,\Lambda$. Then the forward operator $F$ is defined by $F = \tilde{F}\circ P$ with domain ${\cal D}(F) = {\cal D}(P)$.

\subsection{Conditional stability and cross-sensitivity of diffusion and advection coefficients}\label{sec:unique}

Although the inverse problem ultimately concerns the structured coefficients in \eqref{eq:para}, the estimates below are established for the general coefficient-to-solution map $\tilde F$. They therefore describe intrinsic analytical properties of the diffusion-advection equation rather than a particular reconstruction procedure. We first isolate the information carried by the internal density about each coefficient when the other coefficient is fixed.

\begin{theorem}[Conditional stability for the advection coefficient]
    Let $D\in C^1([0,T];W^{1,\infty}(\Omega))$ satisfy
    \begin{equation*}
        D\ge\lambda\ \text{in}\ Q_T,
        \qquad
        \|D\|_{C^1([0,T];W^{1,\infty}(\Omega))}\le A,
    \end{equation*}
    for some positive constants $\lambda$ and $A$. For $i=1,2$, let $K^i\in C^1([0,T];W^{1,\infty}(\Omega))$ satisfy
    \begin{equation*}
        \|K^i\|_{C^1([0,T];W^{1,\infty}(\Omega))}\le A,
    \end{equation*}
    and let $\rho^i=\tilde F(D,K^i)$. If $\rho^1\ge c_\rho>0$ in $Q_T$, then for every $t\in[0,T]$,
    \begin{equation}\label{eq:weighted_K}
        \int_\Omega |K^1(t)-K^2(t)||\nabla\Phi|^2\,dx
        \le \frac{C}{c_\rho}\left(
        \|\partial_t\rho^1(t)-\partial_t\rho^2(t)\|_{L^2(\Omega)}
        +\|\rho^1(t)-\rho^2(t)\|_{H^2(\Omega)}
        \right),
    \end{equation}
    where $C=C(A,\|\nabla\Phi\|_{W^{1,\infty}(\Omega)},\Omega)$. 
    In particular, if $|\nabla\Phi|\ge c_\Phi>0$ in a subdomain $\Omega'\subset\Omega$, then
    \begin{equation}\label{eq:local_stab_K}
        \|K^1(t)-K^2(t)\|_{L^1(\Omega')}
        \le \frac{C}{c_\rho c_\Phi^2}\left(
        \|\partial_t\rho^1(t)-\partial_t\rho^2(t)\|_{L^2(\Omega)}
        +\|\rho^1(t)-\rho^2(t)\|_{H^2(\Omega)}
        \right).
    \end{equation}
\end{theorem}

\begin{proof}
    Set $r\coloneqq\rho^1-\rho^2$ and $q_K\coloneqq K^1-K^2$. Subtracting the two equations of $\rho^1$ and $\rho^2$ gives
    \begin{equation}
        -\div(\rho^1 q_K\nabla\Phi)
        =\partial_t r-\div(D\nabla r)+\div(rK^2\nabla\Phi)
        \eqqcolon f_K
        \quad\text{in }Q_T.
    \end{equation}
    Denote the sign function by
    \begin{equation*}
        \sgn(s) = \begin{cases}
            1, & s>0, \\ 0, & s=0, \\ -1, & s<0.
        \end{cases}
    \end{equation*}
    By the Sobolev chain rule,
    \begin{align}
        -\div(\rho^1|q_K|\nabla\Phi) &= -\nabla|q_K|\cdot\rho^1\nabla\Phi - |q_K|\div(\rho^1\nabla\Phi) \nonumber \\
        &= -\sgn(q_K)[\nabla q_K\cdot\rho^1\nabla\Phi + q_K\div(\rho^1\nabla\Phi)] \nonumber \\
        &= -\sgn(q_K)\div(\rho^1 q_K\nabla\Phi)
        = \sgn(q_K)f_K \quad \text{in } Q_T. \label{eq:diff_K}
    \end{align}
    Set $\psi\coloneqq\Phi-|\Omega|^{-1}\int_\Omega\Phi\,dx$. Multiplying the equation \eqref{eq:diff_K} by $\psi$, integrating by parts over $\Omega$, and using the Poincar\'e inequality to control $\psi$ in terms of $\nabla\Phi$, we obtain
    \begin{align*}
        \int_\Omega \rho^1(t)|q_K||\nabla\Phi|^2\,dx
        &\le \|\psi\|_{L^2(\Omega)} \|f_K(t)\|_{L^2(\Omega)} \\
        &\le C(A,\|\nabla\Phi\|_{W^{1,\infty}(\Omega)},\Omega)
        \left(
        \|\partial_t r(t)\|_{L^2(\Omega)}
        +\|r(t)\|_{H^2(\Omega)}
        \right).
    \end{align*}
    The lower bound on $\rho^1$ gives \eqref{eq:weighted_K}, and \eqref{eq:local_stab_K} follows from the lower bound on $|\nabla\Phi|$ in $\Omega'$.
\end{proof}

\begin{theorem}[Conditional stability for the diffusion coefficient]
    Let $K\in C^1([0,T];W^{1,\infty}(\Omega))$ satisfy
    \begin{equation*}
        \|K\|_{C^1([0,T];W^{1,\infty}(\Omega))}\le A.
    \end{equation*}
    For $i=1,2$, let $D^i\in C^1([0,T];W^{1,\infty}(\Omega))$ satisfy
    \begin{equation*}
        D^i\ge\lambda\ \text{in}\ Q_T,
        \qquad
        \|D^i\|_{C^1([0,T];W^{1,\infty}(\Omega))}\le A,
    \end{equation*}
    and let $\rho^i=\tilde F(D^i,K)$. Then, for every $t\in[0,T]$,
    \begin{equation}\label{eq:weighted_D}
        \int_\Omega |D^1(t)-D^2(t)||\nabla\rho^1(t)|^2\,dx \le C\|\rho^1(t)\|_{L^2(\Omega)}\left(
        \|\partial_t\rho^1(t)-\partial_t\rho^2(t)\|_{L^2(\Omega)}
        +\|\rho^1(t)-\rho^2(t)\|_{H^2(\Omega)}
        \right), 
    \end{equation}
    where $C=C(A,\|\nabla\Phi\|_{W^{1,\infty}(\Omega)},\Omega)$.
    Consequently, if $|\nabla\rho^1(t)|\ge c_g>0$ in a subdomain $\Omega'\subset\Omega$, then
    \begin{equation}\label{eq:local_stab_D}
        \|D^1(t)-D^2(t)\|_{L^1(\Omega')}
        \le \frac{C}{c_g^2}\|\rho^1(t)\|_{L^2(\Omega)}\left(
        \|\partial_t\rho^1(t)-\partial_t\rho^2(t)\|_{L^2(\Omega)} +\|\rho^1(t)-\rho^2(t)\|_{H^2(\Omega)}
        \right). 
    \end{equation}
\end{theorem}

\begin{proof}
    Set $r\coloneqq\rho^1-\rho^2$ and $q_D\coloneqq D^1-D^2$. Subtracting the two equations yields
    \begin{equation*}
        \div(q_D\nabla\rho^1)
        =\partial_t r-\div(D^2\nabla r)+\div(rK\nabla\Phi)
        \eqqcolon f_D
        \quad\text{in }Q_T.
    \end{equation*}
    By the Sobolev chain rule,
    \begin{equation*}
        \div(|q_D|\nabla\rho^1)=\sgn(q_D)f_D
        \quad\text{a.e. in }Q_T.
    \end{equation*}
    Multiplying this identity by $\rho^1(t)$ and integrating by parts over $\Omega$ gives
    \begin{equation*}
        \int_\Omega |q_D||\nabla\rho^1|^2\,dx
        \le \|\rho^1(t)\|_{L^2(\Omega)}\|f_D(t)\|_{L^2(\Omega)}.
    \end{equation*}
    The coefficient bounds give \eqref{eq:weighted_D}, and \eqref{eq:local_stab_D} follows immediately.
\end{proof}

These two theorems reveal an asymmetry intrinsic to the model. For the recovery of $K$, the only solution-dependent nondegeneracy factor is $\rho^1$. Its strict positivity is guaranteed after the time translation; the known illumination weight $|\nabla\Phi|^2$ only restricts identification to regions where the illumination gradient does not vanish. By contrast, the estimate for $D$ is weighted by $|\nabla\rho^1|^2$, and the model provides no general positive lower bound for this quantity. Recovering $D$ therefore requires an additional gradient nondegeneracy condition on the measured density.

We next allow both coefficients to vary while requiring the two coefficient pairs to generate the same internal density. The resulting estimates describe how changes in one coefficient must be balanced by changes in the other.

\begin{proposition}[Cross-sensitivity between diffusion and advection coefficients]\label{prop:cross_sensitivity}
    For $i=1,2$, let $D^i,K^i\in C^1([0,T];W^{1,\infty}(\Omega))$ satisfy
    \begin{equation*}
        D^i\ge\lambda\ \text{in}\ Q_T,
        \qquad
        \|D^i\|_{C^1([0,T];W^{1,\infty}(\Omega))},
        \|K^i\|_{C^1([0,T];W^{1,\infty}(\Omega))}\le A,
    \end{equation*}
    for some positive constants $\lambda$ and $A$. If
    \begin{equation*}
        \tilde F(D^1,K^1)=\tilde F(D^2,K^2)=\rho,
    \end{equation*}
    then the following estimates hold for every $t\in[0,T]$.
    \begin{enumerate}[(i)]
        \item If $\rho\ge c_\rho>0$ in $Q_T$, then
        \begin{equation}\label{eq:cross_K}
            \int_\Omega |K^1(t)-K^2(t)||\nabla\Phi|^2\,dx
            \le \frac{2}{c_\rho} \|\rho(t)\|_{H^2(\Omega)} \|D^1(t)-D^2(t)\|_{W^{1,\infty}(\Omega)}.
        \end{equation}

        \item Conversely,
        \begin{equation}\label{eq:cross_D}
            \int_\Omega |D^1(t)-D^2(t)||\nabla\rho(t)|^2\,dx 
            \le 2\|\rho(t)\|_{H^1(\Omega)}^2 \|\nabla\Phi\|_{W^{1,\infty}(\Omega)} \|K^1(t)-K^2(t)\|_{W^{1,\infty}(\Omega)}.
        \end{equation}
    \end{enumerate}
\end{proposition}

\begin{proof}
    Set $q_D\coloneqq D^1-D^2$ and $q_K\coloneqq K^1-K^2$. Subtracting the equations corresponding to the two coefficient pairs and using the equality of the densities, we obtain the compensation identity
    \begin{equation}\label{eq:compensation_identity}
        \div(\rho q_K\nabla\Phi)
        =\div(q_D\nabla\rho)
        \quad\text{in }Q_T.
    \end{equation}
    To prove \eqref{eq:cross_K}, use the chain rule to rewrite \eqref{eq:compensation_identity} as
    \begin{equation*}
        \div(\rho|q_K|\nabla\Phi)
        =\sgn(q_K)\div(q_D\nabla\rho),
    \end{equation*}
    multiply by $\psi=\Phi-|\Omega|^{-1}\int_\Omega\Phi\,dx$, and integrate by parts. The positivity of $\rho$ and the estimate
    \begin{equation*}
        \|\div(q_D\nabla\rho)(t)\|_{L^2(\Omega)}
        \le 2\|q_D(t)\|_{W^{1,\infty}(\Omega)}\|\rho(t)\|_{H^2(\Omega)}
    \end{equation*}
    give \eqref{eq:cross_K}.
    For the converse estimate, we rewrite \eqref{eq:compensation_identity} as
    \begin{equation*}
        \div(|q_D|\nabla\rho)
        =\sgn(q_D)\div(\rho q_K\nabla\Phi).
    \end{equation*}
    Multiplication by $\rho(t)$ and integration by parts yield
    \begin{align*}
        \int_\Omega |q_D||\nabla\rho|^2\,dx
        &\le \|\rho(t)\|_{L^2(\Omega)}
        \|\div(\rho q_K\nabla\Phi)(t)\|_{L^2(\Omega)} \\
        &\le 2\|\rho(t)\|_{H^1(\Omega)}^2 \|\nabla\Phi\|_{W^{1,\infty}(\Omega)} \|q_K(t)\|_{W^{1,\infty}(\Omega)}.
    \end{align*}
    This is exactly the estimate \eqref{eq:cross_D}.
\end{proof}

These estimates are not tied to a sequential reconstruction procedure. Rather, they quantify the compensation between the diffusion and advection coefficients when fitting the same internal data, thereby illustrating the intrinsic coupling of the simultaneous inverse problem. More precisely, \eqref{eq:compensation_identity} shows that the perturbations of the two transport fluxes balance at the divergence level; it does not assert that the fluxes agree pointwise. This coefficient cross-talk provides an analytical explanation for the strong coupling that makes direct simultaneous minimization challenging in \cref{sec:num}.

The results above describe the spatial coupling between the diffusion and advection coefficients at each observation time. The reduced model \eqref{eq:para}, however, contains an additional structure of a different nature: the transient and steady components are weighted differently in time. In particular, the transient contribution decays exponentially, whereas the steady component becomes increasingly dominant. We next quantify this temporal separation in finite observation intervals.

\subsection{Finite-time separation of the transient and steady regimes}\label{sec:two_stage_estimates}

The exponential structure in \eqref{eq:para} distinguishes two temporal regimes. 
At sufficiently late times, the transient parts $e^{-\gamma t}(D_0-D_s)$ and $e^{-\gamma t}(K_0-K_s)$ are small, so that the dynamics are expected to be well approximated by the steady coefficients $(D_s, K_s)$. 
Conversely, near the initial time, the influence of the steady coefficients enters only through the factor $1-e^{-\gamma t}$. 
We make these two observations quantitative below.

Fix $0<T_1<T_2<T$ and assume that $(D_0,D_s,K_0,K_s,\gamma)\in{\cal D}(F)$. 
We first quantify the accuracy of freezing the time-dependent coefficients at their steady profiles on a late-time interval. 
Let $\rho=F(D_0,D_s,K_0,K_s,\gamma)$ and let $\rho^+$ solve
\begin{equation}\label{eq:frozen_forward}
    \begin{cases}
        \partial_t\rho^+-\div(D_s\nabla\rho^+)+\div(\rho^+K_s\nabla\Phi)=0
        & \text{in }\ \Omega\times(T_2,T), \\
        \partial_\nu\rho^+=0
        & \text{on }\partial\Omega\times(T_2,T), \\
        \rho^+(\cdot,T_2)=\rho(\cdot,T_2)
        & \text{in }\ \Omega.
    \end{cases}
\end{equation}
Note that the full and frozen models start from the same observed density at the beginning of the late-time window.

\begin{proposition}[Late-time frozen-coefficient approximation]\label{prop:frozen_approximation}
    There exists a constant $C=C(\lambda,\Lambda,\Omega,T)>0$, independent of $T_2$ and $\gamma$, such that
    \begin{equation}\label{eq:frozen_error}
        \|\rho-\rho^+\|_{L^2(T_2,T;H^1(\Omega))}
        \le C e^{-\gamma T_2}
        \left(
        \|D_0-D_s\|_{L^\infty(\Omega)}
        +\|(K_0-K_s)\nabla\Phi\|_{L^\infty(\Omega)}
        \right)
        \|\varphi\|_{L^2(\Omega)}. 
    \end{equation}
\end{proposition}

\begin{proof}
    Set $z\coloneqq\rho-\rho^+$. Subtracting the full equation \eqref{eq:adv-diff} and \eqref{eq:frozen_forward}, we obtain
    \begin{equation}\label{eq:frozen_difference}
        \begin{cases}
            \partial_t z-\div(D_s\nabla z)+\div(zK_s\nabla\Phi)
            =\div f^+
            & \text{in }\ \Omega\times(T_2,T), \\
            \partial_\nu z=0
            & \text{on }\partial\Omega\times(T_2,T), \\
            z(\cdot,T_2)=0
            & \text{in }\ \Omega,
        \end{cases}
    \end{equation}
    where
    \begin{equation*}
        f^+(x,t)
        =e^{-\gamma t}\left[(D_0-D_s)\nabla\rho
        -\rho(K_0-K_s)\nabla\Phi\right].
    \end{equation*}
    Noting that $f^+\cdot\nu=0$ on $\partial\Omega$, the energy estimate in \cref{prop:well_posed} gives
    \begin{align*}
        \|z\|_{L^2(T_2,T;H^1(\Omega))}
        &\le C(\lambda,\Lambda,\Omega,T)\|f^+\|_{L^2(\Omega\times(T_2,T))} \\
        &\le C(\lambda,\Lambda,\Omega,T) e^{-\gamma T_2}
        \left(
        \|D_0-D_s\|_{L^\infty(\Omega)}
        +\|(K_0-K_s)\nabla\Phi\|_{L^\infty(\Omega)}
        \right)
        \|\rho\|_{L^2(T_2,T;H^1(\Omega))}.
    \end{align*}
    A second application of \cref{prop:well_posed} to the full solution yields
    \begin{equation*}
        \|\rho\|_{L^2(T_2,T;H^1(\Omega))}
        \le \|\rho\|_{L^2(0,T;H^1(\Omega))}
        \le C(\lambda,\Lambda,\Omega,T)\|\varphi\|_{L^2(\Omega)},
    \end{equation*}
    which proves \eqref{eq:frozen_error}.
\end{proof}

The factor $e^{-\gamma T_2}$ shows that, on a sufficiently late observation window, replacing the time-dependent coefficients by their steady states produces only an exponentially small forward-model error. Thus, late-time observations can be described accurately by a model with steady coefficients $(D_s,K_s)$, without requiring the limiting regime $t\to\infty$.

We next consider the complementary early-time regime. Rewriting \eqref{eq:para} as
\begin{equation*}
    \begin{aligned}
        D(x,t) &= D_0(x) + (1-e^{-\gamma t})(D_s - D_0), \\
        K(x,t) &= K_0(x) + (1-e^{-\gamma t})(K_s - K_0)
    \end{aligned}
\end{equation*}
shows that the dependence of the observed data on the steady coefficients is attenuated near the initial time. We quantify this effect by perturbing $(D_s,K_s)$ while keeping $(D_0,K_0,\gamma)$ fixed. Let
\begin{equation*}
    D_s^r=D_s+\delta D_s\in{\cal T}_+,
    \qquad
    K_s^r=K_s+\delta K_s\in{\cal T},
\end{equation*}
and denote the corresponding time-dependent coefficients by
\begin{align*}
    D^r(x,t)
    &=D_s^r+e^{-\gamma t}(D_0-D_s^r)
    =D(x,t)+(1-e^{-\gamma t})\delta D_s, \\
    K^r(x,t)
    &=K_s^r+e^{-\gamma t}(K_0-K_s^r)
    =K(x,t)+(1-e^{-\gamma t})\delta K_s.
\end{align*}
Let $\rho^r=F(D_0,D_s^r,K_0,K_s^r,\gamma)$ be the associated solution with the same initial state $\varphi$.

\begin{proposition}[Early-time forward sensitivity]\label{prop:stage2_sensitivity}
    There exists a constant $C=C(\lambda,\Lambda,\Omega,T)>0$, independent of $T_1$ and $\gamma$, such that
    \begin{equation}\label{eq:stage2_sensitivity}
        \|\rho^r-\rho\|_{L^2(0,T_1;H^1(\Omega))}
        \le C(1-e^{-\gamma T_1})
        \left(
        \|\delta D_s\|_{L^\infty(\Omega)}
        +\|\delta K_s\nabla\Phi\|_{L^\infty(\Omega)}
        \right)
        \|\varphi\|_{L^2(\Omega)}. 
    \end{equation}
\end{proposition}

\begin{proof}
    Set $w\coloneqq\rho^r-\rho$. Expressing the equation for $\rho$ in terms of the perturbed coefficients $D^r$ and $K^r$ and then subtracting it from the equation for $\rho^r$, we obtain
    \begin{equation*}
        \begin{cases}
            \partial_t w-\div(D^r\nabla w)+\div(wK^r\nabla\Phi)=\div f^-
            & \text{in }\Omega\times(0,T_1), \\
            \partial_\nu w=0
            & \text{on }\partial\Omega\times(0,T_1), \\
            w(\cdot,0)=0
            & \text{in }\Omega,
        \end{cases}
    \end{equation*}
    where
    \begin{equation*}
        f^-(x,t)
        =(1-e^{-\gamma t})
        \left[\delta D_s\nabla\rho
        -\rho\delta K_s\nabla\Phi\right].
    \end{equation*}
    As in the preceding proof, the energy estimate in \cref{prop:well_posed} gives
    \begin{align*}
        \|w\|_{L^2(0,T_1;H^1(\Omega))}
        &\le C(\lambda,\Lambda,\Omega,T)\|f^-\|_{L^2(\Omega\times(0,T_1))} \\
        &\le C(\lambda,\Lambda,\Omega,T) (1-e^{-\gamma T_1})
        \left(
        \|\delta D_s\|_{L^\infty(\Omega)}
        +\|\delta K_s\nabla\Phi\|_{L^\infty(\Omega)}
        \right)
        \|\rho\|_{L^2(0,T_1;H^1(\Omega))} \\
        &\le C(\lambda,\Lambda,\Omega,T) (1-e^{-\gamma T_1})
        \left(
        \|\delta D_s\|_{L^\infty(\Omega)}
        +\|\delta K_s\nabla\Phi\|_{L^\infty(\Omega)}
        \right)\|\varphi\|_{L^2(\Omega)}.
    \end{align*}
    This proves \eqref{eq:stage2_sensitivity}.
\end{proof}

This is a forward sensitivity estimate, rather than an inverse stability result. It shows that perturbations of the steady coefficients affect the early-time solution linearly, with an additional attenuation factor $1-e^{-\gamma T_1}$. 
Hence the early-time dynamics are comparatively insensitive to moderate inaccuracies in $(D_s, K_s)$, provided that the observation window remains within the transient regime.

Together, \cref{prop:frozen_approximation,prop:stage2_sensitivity} reveal a finite-time separation of information: late-time observations are accurately described by the steady coefficients, whereas early-time observations are comparatively insensitive to perturbations of those coefficients. 
This temporal structure will motivate the reconstruction strategy introduced in \cref{sec:num}.

\subsection{Long-time asymptotics and identification of the relaxation rate}\label{sec:long_time_gamma}

The preceding subsection concerns the separation of the transient and steady regimes in a finite observation window. 
A different question is whether the common relaxation rate $\gamma$ within the reduced structure \eqref{eq:para} can be identified from the density evolution. 
For this purpose, finite-time approximation alone is insufficient, and we turn to the genuine long-time dynamics of the reduced model.

We first establish exponential convergence to a steady state $\rho_s$. Once the steady coefficients $(D_s,K_s)$ are known, this asymptotic expansion allows us to extract $\gamma$ from the residual relative to the steady-state operator.

\begin{lemma}\label{lem:steady}
    For every $m>0$, there exists a unique weak solution $\rho_s\in H^1(\Omega)$ of
    \begin{equation}\label{eq:steady}
        \begin{cases}
            -\div(D_s\nabla\rho_s) + \div(\rho_s K_s\nabla\Phi) = 0 & \text{in}\ \Omega, \\
            \partial_\nu \rho_s = 0 & \text{on}\ \partial\Omega, \\
            \rho_s > 0,\ \int_\Omega \rho_s\,dx = m.
        \end{cases}
    \end{equation}
    Moreover, $\rho_s\in H^2(\Omega)$ and hence, $\rho_s\in C(\overline{\Omega})$.
\end{lemma}

\begin{proof}
    Define the operator
    \begin{equation*}
        {\cal L}_s u \coloneqq -\div(D_s\nabla u) + \div(u K_s\nabla\Phi) = -D_s\Delta u + (K_s\nabla\Phi - \nabla D_s)\cdot\nabla u + \div(K_s\nabla\Phi)u
    \end{equation*}
    with homogeneous Neumann boundary condition. Choosing $\mu\ge \|\div(K_s\nabla\Phi)\|_{L^\infty(\Omega)}$, the operator ${\cal L}_s + \mu$ has a nonnegative zero-order term. Then by \cite[Section 6.5.2]{Evans2010}, the operator ${\cal L}_s + \mu$ has a principal eigenvalue $r>0$, which is simple, with a positive eigenfunction $w\in H^1(\Omega)$. That is,
    \begin{equation*}
        \begin{cases}
            {\cal L}_s w = (r - \mu)w & \text{in}\ \Omega, \\
            \partial_\nu w = 0 & \text{on}\ \partial\Omega.
        \end{cases}
    \end{equation*}
    Integrating the equation over $\Omega$ and using the homogeneous Neumann boundary condition yield that
    \begin{equation*}
        0 = \int_\Omega {\cal L}_s w\,dx = (r - \mu)\int_\Omega w\,dx.
    \end{equation*}
    Since $w$ is positive, $\int_\Omega w\,dx > 0$. So $r - \mu = 0$, and therefore ${\cal L}_s w = 0$ in $\Omega$. Finally, normalizing by the prescribed mass, we get a unique positive steady state
    \begin{equation*}
        \rho_s = \frac{m}{\int_\Omega w\,dx} w \in H^1(\Omega)
    \end{equation*}
    satisfying \eqref{eq:steady}. Finally, since $D_s,K_s\nabla\Phi\in W^{1,\infty}(\Omega)$ and $\partial\Omega\in C^2$, standard elliptic regularity for the Neumann problem implies that $\rho_s\in H^2(\Omega)$, and thus $\rho_s\in C(\overline{\Omega})$ by Sobolev's embedding.
\end{proof}

\begin{proposition}\label{prop:rate}
    Let $\rho_s\in H^2(\Omega)$ be the unique solution of \eqref{eq:steady}.
    Then there exist positive constants $C,\alpha$ and $T_*$, depending on $\lambda,A,\Omega,\gamma,\rho_s$, such that
    \begin{equation}\label{eq:rate}
        \|\rho(t) - \rho_s\|_{H^1(\Omega)} \le C e^{-\alpha t}, \qquad t\ge T_*.
    \end{equation}
\end{proposition}
\begin{proof}
    In this proof, $C$ denotes a generic positive constant that may change from line to line. Let
    \begin{equation*}
        d\coloneqq D_0-D_s, \quad k\coloneqq K_0-K_s, \quad {\cal B}w \coloneqq \div(d\nabla w) - \div(wk\nabla\Phi).
    \end{equation*}
    Then from the equation \eqref{eq:adv-diff} with coefficients \eqref{eq:para} and the steady-state equation \eqref{eq:steady}, the difference $u\coloneqq \rho - \rho_s$ satisfies the equation
    \begin{equation}\label{eq:diff}
        \partial_t u + {\cal L}_s u = e^{-\gamma t}{\cal B}\rho \quad \text{in}\ \Omega\times[0,\infty)
    \end{equation}
    and
    \begin{equation*}
        \int_\Omega u(x,t)\,dx = 0, \quad \forall\,t\ge 0
    \end{equation*}
    by the mass conservation of $\rho$. In the following, we first prove an $L^2$ convergence estimate by a weighted energy method and then upgrade it to an $H^1$ estimate by standard parabolic regularity.

    By \cref{prop:high_reg} we know that for all $t\ge 0$, $\rho(t)\in H^2(\Omega)$, and thus $u(t)\in H^2(\Omega)$.
    By the positivity of $\rho_s\in C(\overline{\Omega})$, we write $v(t) = \rho_s^{-1} u(t)\in H^2(\Omega)$ and define the weighted energy terms
    \begin{equation*}
        E_0(t) \coloneqq \frac{1}{2}\int_\Omega v(t)^2\rho_s\,dx, \quad
        E_1(t) \coloneqq \frac{1}{2}\int_\Omega |\nabla v(t)|^2\rho_s\,dx,
    \end{equation*}
    which are equivalent to the usual $L^2$ ones. Then
    \begin{align}
        E_0'(t) = \int_\Omega \partial_t u(t)v(t)\,dx
        &= \int_\Omega (e^{-\gamma t}{\cal B}\rho - {\cal L}_s u)v\,dx \notag \\
        &= -e^{-\gamma t} \int_\Omega (d\nabla\rho - \rho k\nabla\Phi)\cdot\nabla v\,dx - \int_\Omega v{\cal L}_s u\,dx. \label{eq:derivative}
    \end{align}
    To compute the last term, denote the steady-state flux $J_s \coloneqq D_s\nabla\rho_s - \rho_s K_s\nabla\Phi$, which satisfies
    \begin{equation}\label{eq:flux}
        \begin{cases}
            \div J_s = 0 & \text{in}\ \Omega, \\
            J_s\cdot\nu = 0 & \text{on}\ \partial\Omega.
        \end{cases}
    \end{equation}
    Since
    \begin{equation*}
        D_s\nabla u - u K_s\nabla\Phi = D_s\nabla(\rho_s v) - v\rho_s K_s\nabla\Phi = \rho_s D_s\nabla v + v J_s,
    \end{equation*}
    we have
    \begin{align}
        \int_\Omega v{\cal L}_s u\,dx
        &= \int_\Omega \rho_s D_s|\nabla v|^2\,dx + \int_\Omega vJ_s\cdot\nabla v\,dx \notag \\
        &= \int_\Omega \rho_s D_s|\nabla v|^2\,dx + \frac{1}{2} \int_\Omega J_s\cdot\nabla(v^2)\,dx 
        = \int_\Omega \rho_s D_s|\nabla v|^2\,dx, \label{eq:dissipation}
    \end{align}
    where the last equality holds due to \eqref{eq:flux}. Substituting \eqref{eq:dissipation} into \eqref{eq:derivative} yields
    \begin{equation}\label{eq:derivative2}
        E_0'(t) = -e^{-\gamma t}\int_\Omega (d\nabla \rho - \rho k\nabla\Phi)\cdot\nabla v\,dx - \int_\Omega \rho_s D_s|\nabla v|^2\,dx.
    \end{equation}

    Moreover, noting that $\rho = u + \rho_s = \rho_s(v+1)$, we have $\nabla\rho = (v+1)\nabla\rho_s + \rho_s\nabla v$. Thus,
    \begin{equation*}
        E_0'(t) = -e^{-\gamma t}\int_\Omega (v+1)(d\nabla\rho_s - \rho_s k\nabla\Phi)\cdot\nabla v\,dx - \int_\Omega \rho_s D|\nabla v|^2\,dx.
    \end{equation*}
    By H\"older's inequality and Sobolev's embedding,
    \begin{align*}
        \left|\int_\Omega (d\nabla\rho_s - \rho_s k\nabla\Phi)\cdot\nabla v\,dx\right| &\le C E_1(t)^{1/2}, \\
        \left|\int_\Omega v(d\nabla\rho_s - \rho_s k\nabla\Phi)\cdot\nabla v\,dx\right| &\le C\big(\|v\|_{L^4(\Omega)} \|\nabla\rho_s\|_{L^4(\Omega)} + \|v\|_{L^2(\Omega)}\|\rho_s\|_{L^\infty(\Omega)}\big)\|\nabla v\|_{L^2(\Omega)} \\
        &\le C\|v\|_{H^1(\Omega)} \|\rho_s\|_{H^2(\Omega)} \|\nabla v\|_{L^2(\Omega)} \\
        &\le C\big(E_0(t) + E_1(t)\big)^{1/2} E_1(t)^{1/2}.
    \end{align*}
    Since
    \begin{equation*}
        \int_\Omega v(t)\rho_s\,dx = \int_\Omega u(t)\,dx = 0,
    \end{equation*}
    the weighted Poincar\'e inequality reads $E_0(t) \le C_P E_1(t)$, where $C_P$ depends on $\rho_s$ and $\Omega$. Therefore, $D\ge\lambda$ and Young's inequality imply that
    \begin{align*}
        E_0'(t) &\le C_1 e^{-\gamma t} E_1(t)^{1/2} + C_1 e^{-\gamma t} E_1(t) - 2\lambda E_1(t) \\
        &\le \frac{\lambda}{2} E_1(t) + C_2 e^{-2\gamma t} + C_1 e^{-\gamma t} E_1(t) - 2\lambda E_1(t).
    \end{align*}
    Taking $T_0>0$ such that $C_1 e^{-\gamma T_0}\le \lambda/2$ and using the weighted Poincar\'e inequality yield
    \begin{equation*}
        E_0'(t) \le C_2 e^{-2\gamma t} - \lambda E_1(t) \le C_2 e^{-2\gamma t} - \lambda C_P^{-1} E_0(t), \quad \forall\,t\ge T_0.
    \end{equation*}
    Set $\beta=\lambda C_P^{-1}$. Gronwall's inequality on $[T_0,t]$ gives
    \begin{equation*}
        E_0(t)
        \le e^{-\beta(t-T_0)}E_0(T_0)
        +C_2\int_{T_0}^t e^{-\beta(t-s)}e^{-2\gamma s}\,ds
        \le C e^{-2\alpha_0 t}, \qquad t\ge T_0,
    \end{equation*}
    for any fixed $\alpha_0$ satisfying $0<\alpha_0<\min\{\gamma,\beta/2\}$. This proves that
    \begin{equation}\label{eq:rate0}
        \|u(t)\|_{L^2(\Omega)} \le C e^{-\alpha_0 t}, \quad \forall\,t\ge T_0.
    \end{equation}

    Finally, we use standard parabolic regularity to upgrade the $L^2$ convergence rate to an $H^1$ estimate. Note that $u$ satisfies the equation
    \begin{equation}\label{eq:diff2}
        \partial_t u - \div(D(x,t)\nabla u) + \div(u K(x,t)\nabla\Phi) = e^{-\gamma t}{\cal B}\rho_s \quad \text{in}\ \Omega\times[0,\infty).
    \end{equation}
    For each $\tau\ge T_0+2$, choose a temporal cut-off function $\eta\in C^\infty([\tau-3,\tau])$ satisfying
    \begin{equation*}
        0\le\eta(t)\le1,\qquad
        \eta(t)=0 \text{ on }[\tau-3,\tau-2],\qquad
        \eta(t)=1 \text{ on }[\tau-1,\tau]
    \end{equation*}
    and $|\eta'(t)| \le C$ for all $t\in[\tau-3,\tau]$. Set $w=\eta u$. It follows from \eqref{eq:diff2} that $w$ solves
    \begin{equation}
        \partial_t w - \div(D(x,t)\nabla w) + \div(w K(x,t)\nabla\Phi) = \eta'(t)u + \eta(t)e^{-\gamma t}{\cal B}\rho_s \quad \text{in}\ \Omega\times[\tau-3,\tau].
    \end{equation}
    By the standard parabolic regularity estimate \cite[Section 7.1]{Evans2010}, we have
    \begin{align*}
        \|w\|_{C([\tau-3,\tau];H^1(\Omega))}^2 &\le C\|\eta'(t)u + \eta(t)e^{-\gamma t}{\cal B}\rho_s\|_{L^2(\Omega\times[\tau-3,\tau])}^2 \\
        &\le C\left(\int_{\tau-2}^{\tau-1} \|u(t)\|_{L^2(\Omega)}^2\,dt + \|{\cal B}\rho_s\|_{L^2(\Omega)}^2 \int_{\tau-3}^\tau e^{-2\gamma t}\,dt\right).
    \end{align*}
    Since $\rho_s\in H^2(\Omega)$ by \cref{lem:steady} and $\cal B$ is a second-order differential operator, $\|{\cal B}\rho_s\|_{L^2(\Omega)}$ is finite. Thus, from \eqref{eq:rate0} we derive
    \begin{equation*}
        \|u(\tau)\|_{H^1(\Omega)}^2 = \|w(\tau)\|_{H^1(\Omega)}^2 \le C e^{-2\alpha \tau}
    \end{equation*}
    with $\alpha = \min\{\alpha_0, \gamma\}$. Taking $T_*=T_0+2$ completes the proof.
\end{proof}

We now identify the relaxation rate when the steady-state operator ${\cal L}_s$ is known. Define the observable residual
\begin{equation*}
    R_s(t) \coloneqq \partial_t \rho + {\cal L}_s \rho = e^{-\gamma t}{\cal B}\rho = e^{-\gamma t}[{\cal B}\rho_s + {\cal B}(\rho - \rho_s)].
\end{equation*}
By \cref{prop:rate} and the boundedness of ${\cal B}\colon H^1(\Omega)\to H^{-1}(\Omega)$, we have
\begin{equation*}
    \|{\cal B}(\rho - \rho_s)\|_{H^{-1}(\Omega)} = O(e^{-\alpha t}) \ \text{as}\ t\to\infty.
\end{equation*}
So
\begin{equation*}
    \|R_s(t)\|_{H^{-1}(\Omega)} = e^{-\gamma t}\big(\|{\cal B}\rho_s\|_{H^{-1}(\Omega)} + O(e^{-\alpha t})\big) \ \text{as}\ t\to\infty.
\end{equation*}
Thus, if ${\cal B}\rho_s\ne 0$ in $H^{-1}(\Omega)$, then
\begin{equation}\label{eq:gamma_identification}
    \gamma = -\lim_{t\to\infty} \frac{1}{t}\log\|R_s(t)\|_{H^{-1}(\Omega)}.
\end{equation}
This yields the following conditional identifiability result.
\begin{theorem}[Identification of the relaxation rate]\label{thm:gamma}
    Assume that ${\cal B}\rho_s\ne 0$ in $H^{-1}(\Omega)$. If $D_s$ and $K_s$ are known and $\rho$ is observed on a long-time interval $(T_*,\infty)$, then the common relaxation rate $\gamma$ is uniquely determined by \eqref{eq:gamma_identification}.
\end{theorem}

\subsection{Differentiability and adjoint formulation}

The forward operator $F$ is nonlinear in all variables. In this subsection, we show that it is nevertheless Fr\'echet differentiable, and we explicitly compute its Fr\'echet derivatives as well as their adjoint operators. With these analytical tools established, the adjoint state method, which is a gradient-based iterative method, can be employed to efficiently solve the inverse problem. In the following analysis, we assume that $\nabla\Phi\in L^\infty(\Omega)$.

\subsubsection{Differentiability of the forward map}

By the chain rule, we only need to prove the differentiability of $\tilde F$ and $P$ and calculate their derivatives. Since it is easy for $P$, we focus on $\tilde F$.

\begin{lemma}
    The map $\tilde F$ is Fr\'echet differentiable, and its Fr\'echet derivatives at $(D,K)\in\Int{\cal D}(\tilde{F})$ are
    \begin{equation*}
        \frac{\delta\tilde F}{\delta D}\colon L^\infty(Q_T)\to L^2(Q_T), \ g\mapsto u, \quad \frac{\delta\tilde F}{\delta K}\colon L^\infty(Q_T)\to L^2(Q_T), \ h\mapsto v
    \end{equation*}
    with $u,v$ solving the following problems respectively:
    \begin{gather}
        \begin{cases}
            \partial_t u - \div(D\nabla u) + \div(u K\nabla\Phi) = \div(g \nabla\rho) & \text{in}\,\ Q_T, \\
            \partial_\nu u = 0 & \text{on}\ \Sigma_T, \\
            u(\cdot,0) = 0 & \text{in}\ \ \Omega,
        \end{cases} \label{eq:u} \\
        \begin{cases}
            \partial_t v - \div(D\nabla v) + \div(v K\nabla\Phi) = -\div(\rho h\nabla\Phi) & \text{in}\,\ Q_T, \\
            \partial_\nu v = 0 & \text{on}\ \Sigma_T, \\
            v(\cdot,0) = 0 & \text{in}\ \ \Omega,
        \end{cases}
    \end{gather}
    where $\rho = \tilde{F}(D,K)$.
\end{lemma}
\begin{proof}
    For every $(D,K)\in\Int{\cal D}(\tilde{F})$, let $\rho_{g,h} = \tilde{F}(D+g, K+h)$ and $r_1 = \rho_{g,h} - \rho$. Then $r_1$ solves the following problem:
    \begin{equation*}
        \begin{cases}
            \partial_t r_1 - \div(D\nabla r_1) + \div(r_1 K\nabla\Phi) = \div(g\nabla\rho_{g,h} - \rho_{g,h} h\nabla\Phi) & \text{in}\,\ Q_T, \\
            \partial_\nu r_1 = 0 & \text{on}\ \Sigma_T, \\
            r_1(\cdot,0) = 0 & \text{in}\ \ \Omega.
        \end{cases}
    \end{equation*}
    By \cref{prop:well_posed} (ii) we have
    \begin{align*}
        \|r_1\|_{L^2(0,T;H^1(\Omega))} &\le C\|g\nabla\rho_{g,h} - \rho_{g,h}h\nabla\Phi\|_{L^2(Q_T)} \\
        &\le C\big(\|g\|_{L^\infty(Q_T)} + \|h\|_{L^\infty(Q_T)}\big)\|\rho_{g,h}\|_{L^2(0,T;H^1(\Omega))} \\
        &\le C\big(\|g\|_{L^\infty(Q_T)} + \|h\|_{L^\infty(Q_T)}\big),
    \end{align*}
    where the constants $C$ depend on $\lambda,\Lambda,\|\nabla\Phi\|_{L^\infty(\Omega)},\Omega,T$ and $\|\varphi\|_{L^2(\Omega)}$.
    Now let $r_2 = \rho_{g,h} - \rho - u - v$. Then $r_2$ solves the following problem:
    \begin{equation*}
        \begin{cases}
            \partial_t r_2 - \div(D\nabla r_2) + \div(r_2 K\nabla\Phi) = \div(g\nabla r_1 - r_1 h\nabla\Phi) & \text{in}\,\ Q_T, \\
            \partial_\nu r_2 = 0 & \text{on}\ \Sigma_T, \\
            r_2(\cdot,0) = 0 & \text{in}\ \ \Omega.
        \end{cases}
    \end{equation*}
    Using \cref{prop:well_posed} (ii) again we have
    \begin{align*}
        \|r_2\|_{L^2(Q_T)} &\le C\|g\nabla r_1 - r_1 h\nabla\Phi\|_{L^2(Q_T)} \\
        &\le C\big(\|g\|_{L^\infty(Q_T)} + \|h\|_{L^\infty(Q_T)}\big)^2,
    \end{align*}
    where the constants $C$ depend on $\lambda,\Lambda,\|\nabla\Phi\|_{L^\infty(\Omega)},\Omega,T$ and $\|\varphi\|_{L^2(\Omega)}$.
    Therefore, $\tilde{F}$ is Fr\'echet differentiable at $(D,K)$ and the Fr\'echet derivatives are given by $g\mapsto u$ and $h\mapsto v$.
\end{proof}

We conclude by the chain rule that
\begin{proposition}
    The map $F$ is Fr\'echet differentiable, and its Fr\'echet derivatives at $(D_0,D_s,K_0,K_s,\gamma)\in\Int{\cal D}(F)$ are given by
    \begin{alignat*}{2}
        & \frac{\delta F}{\delta D_0}\colon L^\infty(\Omega)\to L^2(Q_T),\quad
        && g_0\mapsto \frac{\delta\tilde F}{\delta D}[e^{-\gamma t} g_0], \\
        & \frac{\delta F}{\delta D_s}\colon L^\infty(\Omega)\to L^2(Q_T),
        && g_s\mapsto \frac{\delta\tilde F}{\delta D}[(1-e^{-\gamma t})g_s], \\
        & \frac{\delta F}{\delta K_0}\colon L^\infty(\Omega)\to L^2(Q_T),
        && h_0\mapsto \frac{\delta\tilde F}{\delta K}[e^{-\gamma t} h_0], \\
        & \frac{\delta F}{\delta K_s}\colon L^\infty(\Omega)\to L^2(Q_T),
        && h_s\mapsto \frac{\delta\tilde F}{\delta K}[(1-e^{-\gamma t})h_s], \\
        & \frac{\delta F}{\delta \gamma}\colon \R\to L^2(Q_T),
        && g_\gamma\mapsto -g_\gamma\left(\frac{\delta\tilde F}{\delta D}[t e^{-\gamma t}(D_0-D_s)] + \frac{\delta\tilde F}{\delta K}[t e^{-\gamma t}(K_0-K_s)]\right).
    \end{alignat*}
\end{proposition}

\subsubsection{Adjoint operators}

To calculate the adjoint operators of the Fr\'echet derivatives of $F$, we first derive those of $\tilde{F}$.
\begin{lemma}
    The adjoint operators of the Fr\'echet derivatives of $\tilde{F}$ at $(D,K)\in\Int{\cal D}(\tilde{F})$ are given by
    \begin{equation*}
        \bigg(\frac{\delta\tilde{F}}{\delta D}\bigg)^*\colon L^2(Q_T)\to L^\infty(Q_T)^*,\ \xi\mapsto\nabla\rho\cdot\nabla w_\xi, \quad \bigg(\frac{\delta\tilde{F}}{\delta K}\bigg)^*\colon L^2(Q_T)\to L^\infty(Q_T)^*,\ \eta\mapsto -\rho\nabla\Phi\cdot\nabla w_\eta,
    \end{equation*}
    where $\rho = \tilde{F}(D,K)$, and $w_\xi$ solves the following problem:
    \begin{equation}\label{eq:w}
        \begin{cases}
            \partial_t w_\xi + \div(D\nabla w_\xi) + K\nabla\Phi\cdot\nabla w_\xi = \xi & \text{in}\,\ Q_T, \\
            \partial_\nu w_\xi = 0 & \text{on}\ \Sigma_T, \\
            w_\xi(\cdot,T) = 0 & \text{in}\ \ \Omega.
        \end{cases}
    \end{equation}
\end{lemma}
\begin{proof}
    By the definition of the adjoint operator, for each $\xi\in L^2(Q_T)$ and $g\in L^\infty(Q_T)$ we have
    \begin{equation*}
        \bigg<\bigg(\frac{\delta\tilde{F}}{\delta D}\bigg)^*[\xi], g\bigg> = \bigg<\xi, \frac{\delta\tilde{F}}{\delta D}[g]\bigg> = \langle\xi, u\rangle
    \end{equation*}
    with $u$ solving \eqref{eq:u}. Substituting $\xi$ using \eqref{eq:w}, we have
    \begin{align*}
        \langle\xi, u\rangle &= \int_0^{T} \int_\Omega [\partial_t w_\xi + \div(D\nabla w_\xi) + K\nabla\Phi\cdot\nabla w_\xi]u\,dx\,dt \\
        &= -\int_0^{T} \int_\Omega w_\xi [\partial_t u - \div(D\nabla u) + \div(uK\nabla\Phi)]\,dx\,dt \\
        &= -\int_0^{T} \int_\Omega w_\xi \div(g\nabla\rho)\,dx\,dt = \int_0^{T} \int_\Omega g\nabla\rho\cdot\nabla w_\xi\,dx\,dt,
    \end{align*}
    showing that $(\frac{\delta\tilde{F}}{\delta D})^*[\xi] = \nabla\rho\cdot\nabla w_\xi$. Analogously, it can be proved that $(\frac{\delta\tilde{F}}{\delta K})^*[\eta] = -\rho\nabla\Phi\cdot\nabla w_\eta$.
\end{proof}

By the chain rule again, we conclude that
\begin{proposition}
    The adjoint operators of the Fr\'echet derivatives of $F$ at $(D_0,D_s,K_0,K_s,\gamma)\in\Int{\cal D}(F)$ are given by
    \begin{alignat*}{2}
        & \left(\frac{\delta F}{\delta D_0}\right)^*\colon L^2(Q_T)\to L^\infty(\Omega)^*, \quad 
        && \xi_0\mapsto \int_0^{T} e^{-\gamma t} \nabla\rho\cdot\nabla w_{\xi_0}\,dt, \\
        & \left(\frac{\delta F}{\delta D_s}\right)^*\colon L^2(Q_T)\to L^\infty(\Omega)^*, 
        && \xi_s\mapsto \int_0^{T} (1-e^{-\gamma t}) \nabla\rho\cdot\nabla w_{\xi_s}\,dt, \\
        & \left(\frac{\delta F}{\delta K_0}\right)^*\colon L^2(Q_T)\to L^\infty(\Omega)^*, 
        && \eta_0\mapsto -\int_0^{T} e^{-\gamma t} \rho\nabla\Phi\cdot\nabla w_{\eta_0}\,dt, \\
        & \left(\frac{\delta F}{\delta K_s}\right)^*\colon L^2(Q_T)\to L^\infty(\Omega)^*, 
        && \eta_s\mapsto -\int_0^{T} (1-e^{-\gamma t}) \rho\nabla\Phi\cdot\nabla w_{\eta_s}\,dt, \\
        & \left(\frac{\delta F}{\delta \gamma}\right)^*\colon L^2(Q_T)\to\R, 
        && \xi_\gamma\mapsto \int_0^T \int_\Omega te^{-\gamma t}[-(D_0 - D_s)\nabla\rho + \rho(K_0 - K_s)\nabla\Phi]\cdot\nabla w_{\xi_\gamma}\,dx\,dt,
    \end{alignat*}
    where $\rho = F(D_0,D_s,K_0,K_s,\gamma)$ and $w_\xi$ solves \eqref{eq:w}.
\end{proposition}

\section{Numerical experiments}\label{sec:num}

In this section, we evaluate the performance of the proposed reconstruction
method through numerical experiments. 
The reconstruction is formulated as the minimization of the data misfit
functional
\begin{equation*}
J(D_0,D_s,K_0,K_s,\gamma)
=
\left\|
F(D_0,D_s,K_0,K_s,\gamma)
-\rho(x,t)
\right\|^2_{L^2(Q_T)}
\end{equation*}
with respect to the five unknown parameters
$(D_0,D_s,K_0,K_s,\gamma)$.

\subsection{Two-stage reconstruction method}

The analysis in \cref{sec:ip} suggests two complementary mechanisms for reducing the difficulty of the simultaneous reconstruction. 
\cref{prop:cross_sensitivity} quantifies the intrinsic coupling between the diffusion and advection coefficients, while \cref{prop:frozen_approximation,prop:stage2_sensitivity} show that the structured model naturally separates the steady and transient components across different temporal regimes. 
Motivated by these observations, we reconstruct the five unknown parameters in two stages.

For sufficiently large time, the time-dependent diffusion and advection
coefficients are well approximated by their asymptotic profiles,
\[
D(x,t)\approx D_s(x),
\qquad
K(x,t)\approx K_s(x).
\]
The late-time estimate in \cref{prop:frozen_approximation} makes this approximation quantitative: when the frozen model is initialized by the observed density at $T_2$, its forward error in a late-time observation window
\[
Q^+=\Omega\times[T_2,T]
\]
is of order $e^{-\gamma T_2}$.
This motivates Stage 1, in which the asymptotic coefficients $D_s$ and $K_s$
are reconstructed from late-time observations $\rho|_{Q^+}$. 
We denote by $F_s(D_s,K_s)=\rho^+$
the forward operator associated with the steady-regime problem \eqref{eq:frozen_forward} and minimize
\begin{equation*}
J_1(D_s,K_s)
=
\left\|
F_s(D_s,K_s)-\rho(x,t)
\right\|^2_{L^2(Q^+)}.
\end{equation*}

After obtaining the reconstructed asymptotic coefficients
$(D_s^r,K_s^r)$ from Stage 1, we fix them and return to the original
time-dependent model \eqref{eq:adv-diff}. The early-time observations contain information about
the remaining parameters
\[
(D_0,K_0,\gamma).
\]
In Stage 2, we therefore substitute $(D_s^r,K_s^r)$ into the original forward
model \eqref{eq:adv-diff} and minimize
\begin{equation*}
J_2(D_0,K_0,\gamma)
=
\left\|
F(D_0,D_s^r,K_0,K_s^r,\gamma)
-
\rho(x,t)
\right\|^2_{L^2(Q^-)},
\end{equation*}
where
\[
Q^-=\Omega\times[0,T_1]
\]
denotes the early-time observation window. We assume that
$0<T_1<T_2<T$. No data misfit is evaluated on the intermediate interval
$(T_1,T_2)$, while the observed snapshot at $T_2$ initializes Stage 1. This
choice separates the transient and asymptotic information.
Moreover, \cref{prop:stage2_sensitivity} shows that, with $D_0$, $K_0$, and $\gamma$ fixed, errors in $(D_s^r,K_s^r)$ perturb the early-time forward solution by at most a factor proportional to $1-e^{-\gamma T_1}$. 
Thus, the propagation of reconstruction errors in Stage 1 into the forward-model error in Stage 2 is quantitatively controlled.

For the spatial coefficients, we use the same weighted normalized gradient
iteration in both stages. More precisely,
\[
a\in\{D_s,K_s\}
\]
in Stage 1, whereas
\[
a\in\{D_0,K_0\}
\]
in Stage 2. For each coefficient $a$, the gradient direction is first
weighted by a spatial function $w$:
\[
\widetilde g_a^{(k)}
=
w g_a^{(k)}.
\]
We define the effective update region by
\[
\Omega_u
=
\{x\in\Omega:w(x)>0\}.
\]
The corresponding unprojected trial iterate is computed as
\begin{equation}
\widehat a^{(k+1)}
=
a^{(k)}
+
\eta_a
\frac{
\|a^{(k)}\|_{\ell^2(\Omega_u)}
}{
\|\widetilde g_a^{(k)}\|_{\ell^2(\Omega_u)}
}
\widetilde g_a^{(k)}.
\label{eq:two_stage_update}
\end{equation}

After each update, the spatial coefficient is projected according to
\begin{equation}
a^{(k+1)}
=
w\widehat a^{(k+1)}
+
(1-w)a^\dagger,
\label{eq:two_stage_projection}
\end{equation}
where $a^\dagger$ denotes the known coefficient value in the prescribed
near-boundary region. Thus, the coefficients are updated in the interior
while remaining fixed to the prescribed values near the boundary.

For the scalar parameter $\gamma$, we update it separately according to
\begin{equation}
\gamma^{(k+1)}
=
\gamma^{(k)}
+
\eta_\gamma g_\gamma^{(k)}.
\label{eq:two_stage_gamma}
\end{equation}

For both stages, the stopping criterion is met when either the maximum number of iterations is reached or the data error increases after an iteration.

The complete reconstruction procedure is summarized in
Algorithm~\ref{alg:two_stage}.

\begin{algorithm}[htbp]
\setstretch{1.35}
\caption{Two-stage reconstruction algorithm}
\label{alg:two_stage}

\begin{algorithmic}[1]

\STATE \textbf{Input:}
interior observations $\rho|_{Q_T}$,
late-time window $Q^+$, and early-time window $Q^-$.

\STATE \textbf{Stage 1: Reconstruction of asymptotic coefficients.}

\STATE Initialize $(D_s^{(0)},K_s^{(0)})$
and compute the initial data error $E_1^{(0)}$ on $Q^+$.

\FOR{$k=0,\ldots,N_1-1$}

\STATE Compute $g_{D_s}^{(k)}$ and $g_{K_s}^{(k)}$
for $J_1$ using the observations on $Q^+$.

\STATE Set
$\widetilde g_{D_s}^{(k)}=w g_{D_s}^{(k)}$
and
$\widetilde g_{K_s}^{(k)}=w g_{K_s}^{(k)}$.

\STATE Update $D_s$ and $K_s$ using
\eqref{eq:two_stage_update}.

\STATE Apply the boundary projection
\eqref{eq:two_stage_projection}.

\STATE Compute the updated data error $E_1^{(k+1)}$ on $Q^+$.

\IF{$E_1^{(k+1)}>E_1^{(k)}$}
    \STATE Restore
    $D_s^{(k+1)}=D_s^{(k)}$
    and
    $K_s^{(k+1)}=K_s^{(k)}$.
    \STATE \textbf{break}
\ENDIF

\ENDFOR

\STATE Let $k_1^\ast$ denote the final accepted iteration and set $D_s^r=D_s^{(k_1^\ast)},
\ K_s^r=K_s^{(k_1^\ast)}$.

\STATE \textbf{Stage 2: Reconstruction of the remaining parameters.}

\STATE Fix
$D_s=D_s^r$
and
$K_s=K_s^r$.

\STATE Initialize
$(D_0^{(0)},K_0^{(0)},\gamma^{(0)})$
and compute the initial data error $E_2^{(0)}$ on $Q^-$.

\FOR{$k=0,\ldots,N_2-1$}

\STATE Compute $g_{D_0}^{(k)}$ and $g_{K_0}^{(k)}$
for $J_2$ using the observations on $Q^-$.

\STATE Compute
$g_\gamma^{(k)}$.

\STATE Set
$\widetilde g_{D_0}^{(k)}=w g_{D_0}^{(k)}$
and
$\widetilde g_{K_0}^{(k)}=w g_{K_0}^{(k)}$.

\STATE Update $D_0$ and $K_0$ using
\eqref{eq:two_stage_update}.

\STATE Apply the boundary projection
\eqref{eq:two_stage_projection}.

\STATE Update $\gamma$ using
\eqref{eq:two_stage_gamma}.

\STATE Compute the updated data error $E_2^{(k+1)}$ on $Q^-$.

\IF{$E_2^{(k+1)}>E_2^{(k)}$}
    \STATE Restore
    $D_0^{(k+1)}=D_0^{(k)}$,
    $K_0^{(k+1)}=K_0^{(k)}$,
    and
    $\gamma^{(k+1)}=\gamma^{(k)}$.
    \STATE \textbf{break}
\ENDIF

\ENDFOR

\STATE Let $k_2^\ast$ denote the final accepted iteration.

\RETURN
$\bigl(
D_0^{(k_2^\ast)},
D_s^r,
K_0^{(k_2^\ast)},
K_s^r,
\gamma^{(k_2^\ast)}
\bigr)$.

\end{algorithmic}
\end{algorithm}

\subsection{Numerical results}

We now present the numerical results obtained by the proposed two-stage
reconstruction method. All experiments are performed on the computational
domain
\[
\Omega=(0,400)^2,
\]
which is discretized by a uniform triangulation associated with a
$400\times400$ rectangular grid. The time interval is $[0,1800]$, discretized
using $1200$ uniform time steps.

The light-dependent drift field $\nabla\Phi(x)$ is chosen as
\[
\nabla\Phi(x_1,x_2)
=
\left(
\frac{1}{1000}\,\psi(x_1),
\frac{1}{1000}\,\psi(x_2)
\right),
\]
where
\[
\psi(\tau)
=
\min\!\left(
\max\!\left(\frac{\tau-10}{10},0\right),
\max\!\left(\frac{390-\tau}{10},0\right),
1
\right).
\]

By construction, this vector field is conservative. Indeed, it is the gradient
of the scalar function
\[
\Phi(x_1,x_2)
=
\frac{1}{1000}
\left(
\int_0^{x_1}\psi(\tau)\,d\tau
+
\int_0^{x_2}\psi(\tau)\,d\tau
\right).
\]
Moreover, since $\psi(0)=\psi(400)=0$, we have
\[
\nabla\Phi\cdot\nu=0
\qquad
\text{on }\partial\Omega.
\]
Thus, the prescribed drift field satisfies the assumptions imposed on
$\Phi$.

The initial condition is chosen as
\[
\varphi(x_1,x_2)
=
4-\frac{x_1+x_2}{200}.
\]

Define the Gaussian function $h(x;a,b,c,d)$ by
\begin{equation*}
h(x;a,b,c,d)
\coloneqq
\exp\left[
-\left(\frac{x_1-a}{b}\right)^2
-\left(\frac{x_2-c}{d}\right)^2
\right].
\end{equation*}

The true spatial coefficients used in the numerical experiments are
\begin{align*}
D_0^\dagger(x)
&=
0.4+h(x;200,150,150,100),
&
K_0^\dagger(x)
&=
2\left(
1+2h(x;150,100,200,150)
\right),
\\
D_s^\dagger(x)
&=
0.4+h(x;250,100,250,100),
&
K_s^\dagger(x)
&=
2\left(
1+2h(x;200,150,250,100)
\right).
\end{align*}
The true scalar parameter is
\[
\gamma^\dagger=0.005.
\]

Both the forward problem and the corresponding adjoint equation are solved
using the finite element method implemented in
FEniCSx~\cite{baratta2023dolfinx}.

The reconstruction is initialized by
\begin{align*}
D_0^{(0)}(x)
&=
0.4+0.5h(x;200,150,150,100),
&
K_0^{(0)}(x)
&=
2\left(
1+h(x;150,100,200,150)
\right),
\\
D_s^{(0)}(x)
&=
0.4+0.5h(x;250,100,250,100),
&
K_s^{(0)}(x)
&=
2\left(
1+h(x;200,150,250,100)
\right).
\end{align*}
The initial value of the scalar parameter is
\[
\gamma^{(0)}=0.006.
\]

For the spatial coefficients, we use the weighted normalized gradient update
and boundary treatment described in
\eqref{eq:two_stage_update}--\eqref{eq:two_stage_projection}.
The step parameters are
\[
\eta_{D_0}=0.0052,\quad
\eta_{D_s}=0.0040,\quad
\eta_{K_0}=0.0040,\quad
\eta_{K_s}=0.0048,
\]
and
\[
\eta_\gamma=0.0024.
\]

The smooth weight function $w$ in
\eqref{eq:two_stage_projection} is
\[
w(x)
=
3s(x)^2-2s(x)^3,
\]
where
\[
s(x)
=
\min
\left\{
1,
\max
\left\{
0,
\frac{d_{\partial\Omega}(x)}{40}
\right\}
\right\}.
\]
Thus, the gradients are fully used in the interior square $[40,360]^2$ and
gradually vanish near the boundary. The reconstructed coefficients are
blended with their prescribed near-boundary values according to
\eqref{eq:two_stage_projection}. In addition, the diffusion coefficients are
projected onto the admissible set
\[
D_0,D_s\geq\lambda,
\qquad
\lambda=0.4.
\]

For Stage 1, we use the late-time observations
\[
Q^+
=
\Omega\times[1200,1800]
\]
with $T_2=1200$ to reconstruct the asymptotic coefficients $D_s$ and $K_s$.
The first stage is terminated after 107 iterations, while the maximum number of iterations is set to 120.

After obtaining $D_s^r$ and $K_s^r$, these coefficients are fixed in Stage 2.
We then use the early-time observations
\[
Q^-
=
\Omega\times[0,480]
\]
with $T_1=480$ to reconstruct $D_0$, $K_0$, and $\gamma$. The second stage is performed for 70 iterations, reaching the prescribed maximum number of iterations. In each stage, the forward and adjoint problems are solved
only over the corresponding observation interval.

The reconstruction results are shown in
Figure~\ref{fig:two_stage_results}. To evaluate the convergence of the spatial
coefficients, we use the coefficient error normalized by the corresponding
initial error:
\begin{equation*}
\mathcal E_a^{(k)}
=
\frac{
\|a^{(k)}-a^\dagger\|_{L^2(\Omega)}
}{
\|a^{(0)}-a^\dagger\|_{L^2(\Omega)}
},
\qquad
a\in\{D_0,D_s,K_0,K_s\}.
\end{equation*}

Since different observation windows are used in the two stages, the normalized
data misfit is evaluated separately on the corresponding time intervals. In
Stage 1, it is defined by
\begin{equation*}
\mathcal E_{\mathrm{data},1}^{(k)}
=
\frac{
\|F_s(D_s^{(k)},K_s^{(k)})-\rho\|_{L^2(Q^+)}
}{
\|F_s(D_s^{(0)},K_s^{(0)})-\rho\|_{L^2(Q^+)}
}.
\end{equation*}

In Stage 2, the normalized data misfit is defined by
\begin{equation*}
\mathcal E_{\mathrm{data},2}^{(k)}
=
\frac{
\left\|
F(D_0^{(k)},D_s^r,K_0^{(k)},K_s^r,
\gamma^{(k)})
-\rho
\right\|_{L^2(Q^-)}
}{
\left\|
F(D_0^{(0)},D_s^r,K_0^{(0)},K_s^r,
\gamma^{(0)})
-\rho
\right\|_{L^2(Q^-)}
}.
\end{equation*}

To quantify the final reconstruction accuracy, we compute the relative
$L^2$ error
\begin{equation*}
E_a
=
\frac{
\|a-a^\dagger\|_{L^2(\Omega)}
}{
\|a^\dagger\|_{L^2(\Omega)}
},
\qquad
a\in\{D_0,D_s,K_0,K_s\}.
\end{equation*}

The final relative errors of the reconstructed spatial coefficients are
\[
E_{D_0}=0.0996,
\qquad
E_{D_s}=0.0702,
\]
and
\[
E_{K_0}=0.0759,
\qquad
E_{K_s}=0.0639.
\]
The reconstructed scalar parameter is
\[
\gamma=0.00519.
\]

The numerical results show that the proposed two-stage strategy provides
accurate reconstructions of all five unknown parameters. The relative errors
of the four spatial coefficients are below $10\%$, while the reconstructed
relative error of $\gamma$ is less than $4\%$. These results indicate that
separating the asymptotic and transient information provides an effective
way to reduce the difficulty of the coupled reconstruction problem.

\begin{figure}[p]
\centering

\includegraphics[width=0.99\textwidth]
{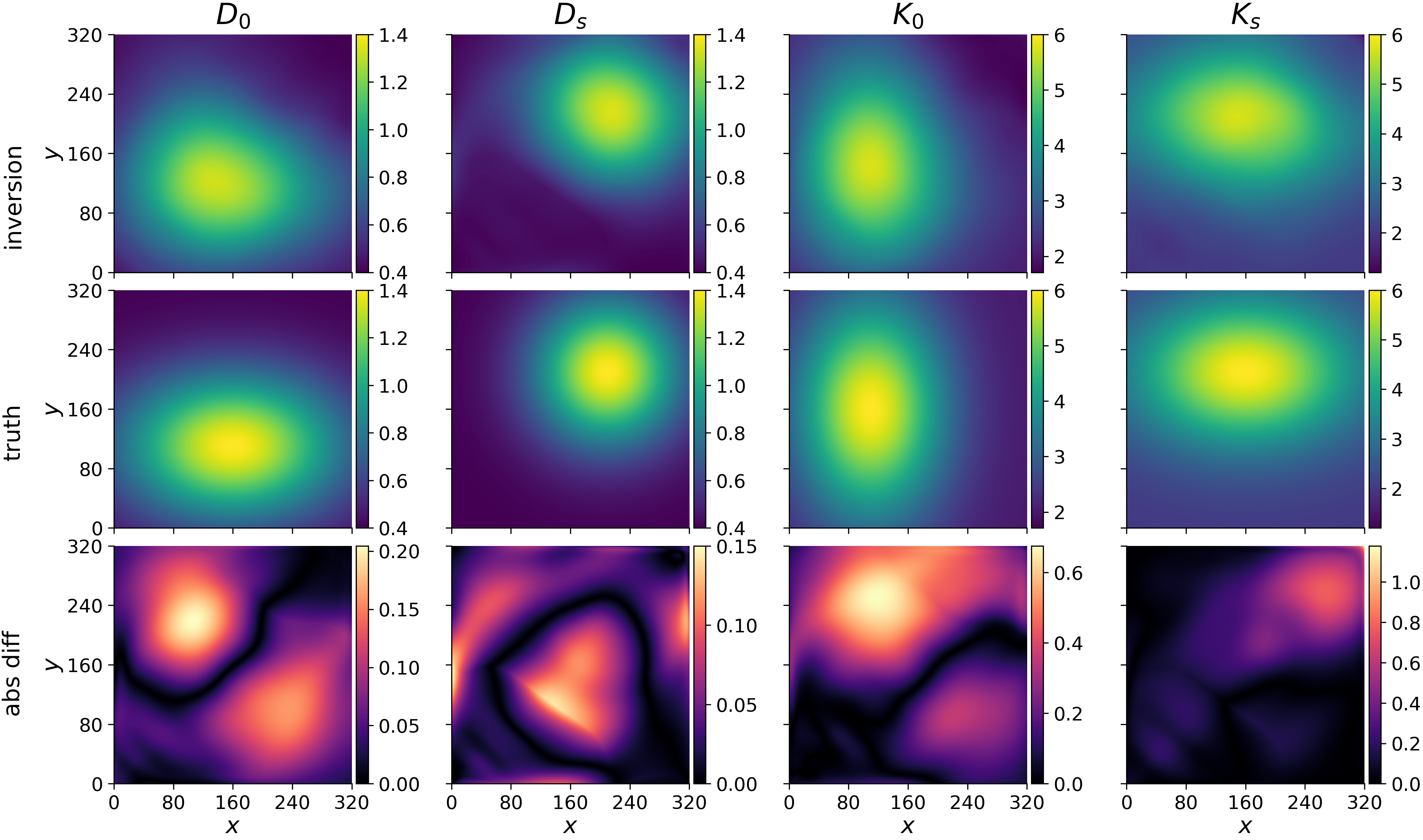}

\vspace{0.3em}

{\small
Reconstructed coefficients, exact coefficients, and absolute errors.
\par}

\vspace{0.7em}

\begin{minipage}[t]{0.485\textwidth}
    \centering
    \includegraphics[width=\linewidth]
    {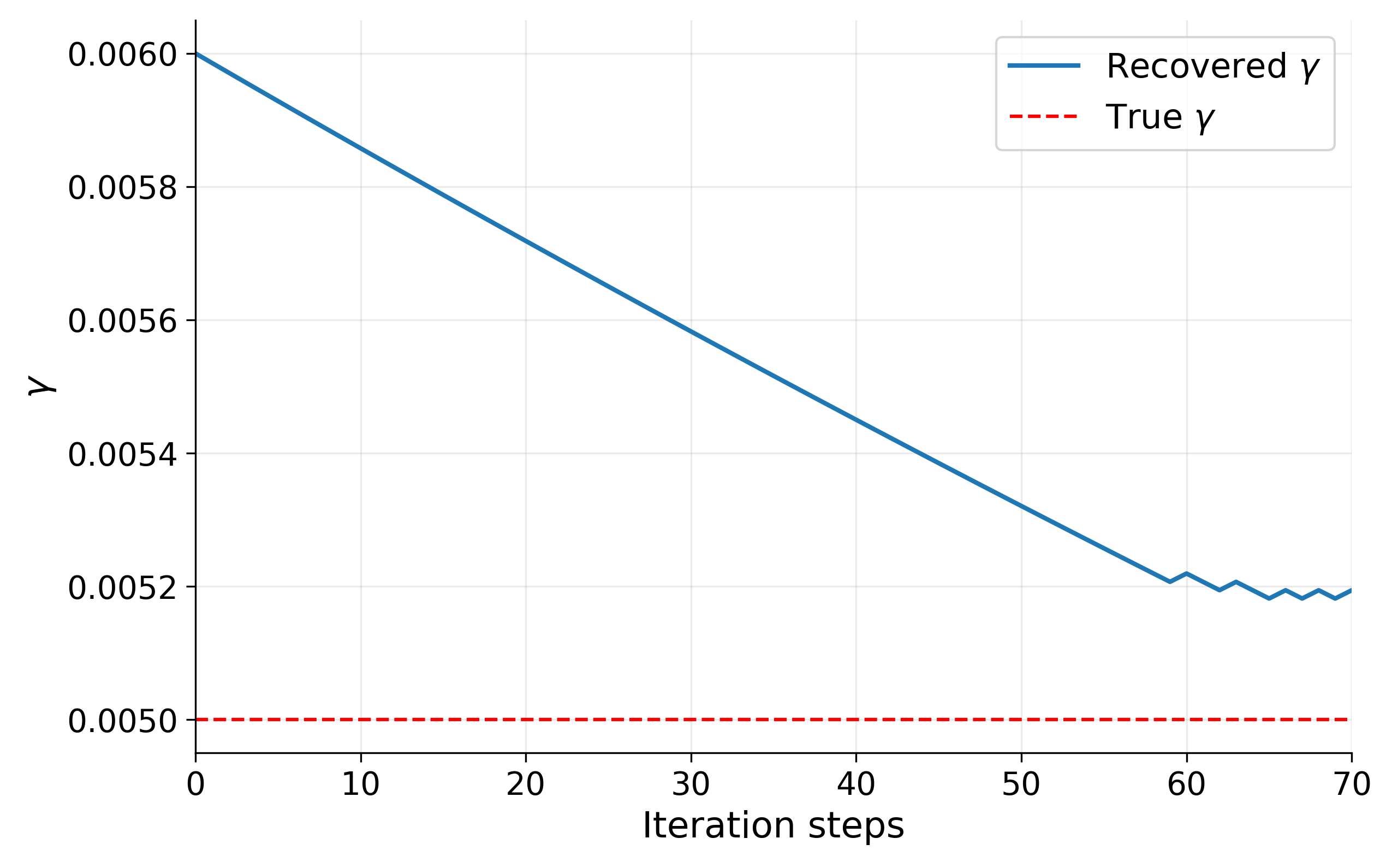}

    \vspace{0.3em}

    {\small
    Reconstruction of $\gamma$.
    \par}
\end{minipage}
\hfill
\begin{minipage}[t]{0.485\textwidth}
    \centering
    \includegraphics[width=\linewidth]
    {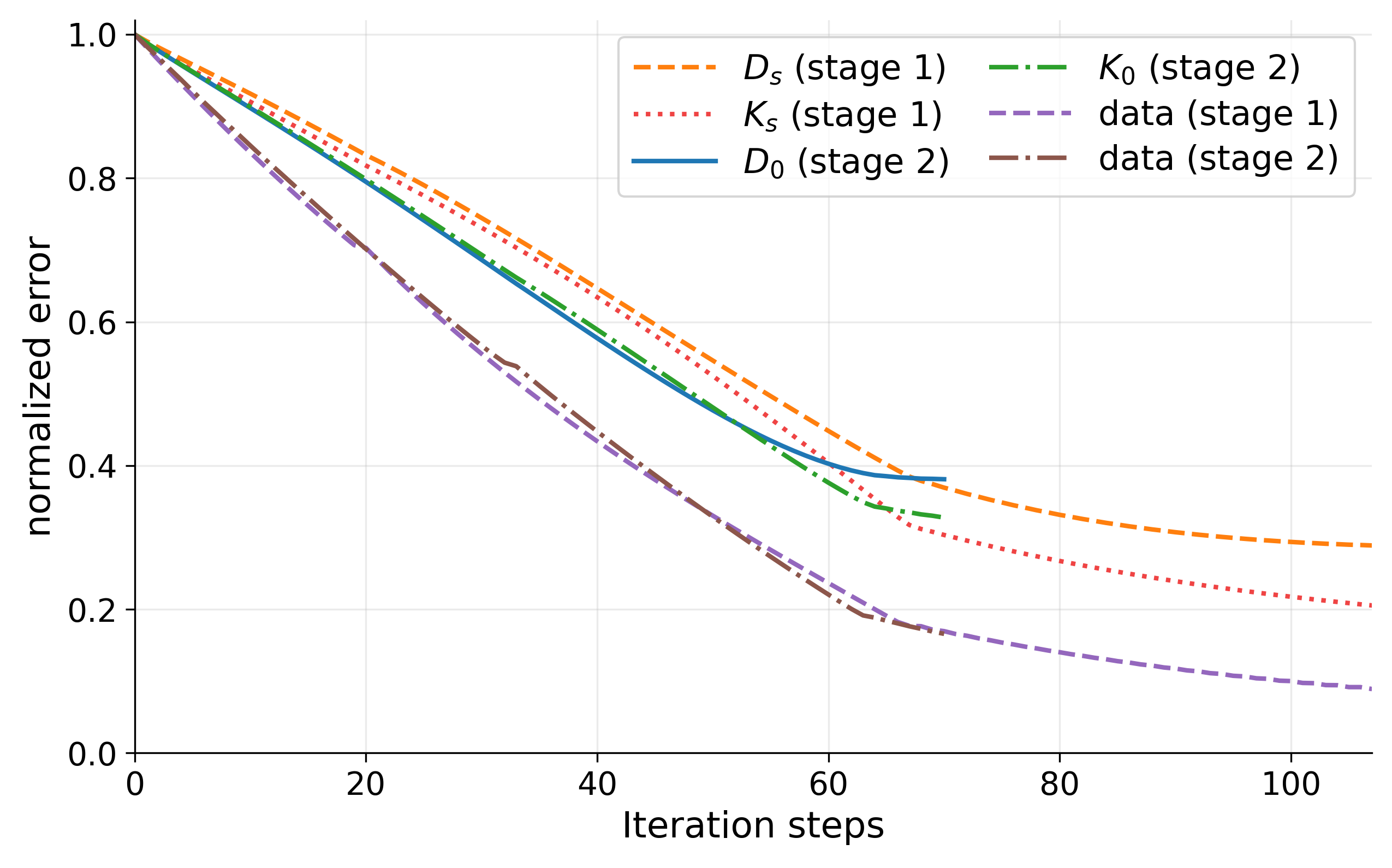}

    \vspace{0.3em}

    {\small
    Normalized coefficient errors and normalized data misfit.
    \par}
\end{minipage}

\caption{%
Numerical results obtained using the two-stage reconstruction strategy.
}
\label{fig:two_stage_results}

\end{figure}
\section{Conclusion}\label{sec:con}
In this work, we investigated an inverse problem of a phototaxis model of Keller--Segel type, with the aim of recovering the spacetime-dependent diffusion and advection coefficients from internal measurements of the population density. 
By exploiting the adaptation mechanism of the model, we considered a reduced coefficient structure in which the diffusion coefficient $D(x,t)$ and the scalar advection coefficient $K(x,t)$ relax exponentially toward their steady states.

On the theoretical side, we established conditional Lipschitz stability estimates for recovering $K$ when $D$ is known and, conversely, for recovering $D$ when $K$ is prescribed.
The two estimates exhibit different nondegeneracy mechanisms: recovery of $K$ benefits from strict positivity of the measured density, whereas recovery of $D$ requires additional control of its spatial gradient.
We further derived cross-sensitivity estimates showing how perturbations of the diffusion and advection coefficients must compensate when two coefficient pairs generate the same internal density, thereby quantifying the intrinsic coupling of the simultaneous inverse problem.

Considering the exponential coefficient structure, we identify a quantitative separation of temporal regimes on finite observation intervals. 
At late times, the full solution is approximated by the steady-coefficient model with an exponentially small error, whereas at early times the influence of perturbations in the steady coefficients is attenuated by a transient factor.
These estimates provide a model-based explanation for separating steady and transient information in time. 
In addition, we established exponential convergence of the density toward its steady state and proved that, once the steady-state coefficients are known, the relaxation rate can be uniquely identified from the long-time decay of an observable residual.
We also derived the Fr\'echet derivatives of the forward map and their corresponding adjoint operators.

Motivated by this analytical structure, we developed a two-stage gradient-based reconstruction strategy.
Late-time observations are first used to recover the steady coefficients $D_s$ and $K_s$, after which early-time observations are used to reconstruct the transient profiles $D_0, K_0$, and the relaxation rate $\gamma$.
Numerical experiments demonstrated the feasibility of this approach and showed that all the parameters can be reconstructed with good accuracy in the examples considered.

Several questions remain open. 
A natural next step is to establish stability estimates for the simultaneous recovery of the diffusion and advection parameters without assuming that either coefficient is known a priori. 
It would also be desirable to obtain a more quantitative understanding of the conditioning of the two-stage reconstruction, including the choice of temporal windows and the propagation of measurement noise. 
Finally, validating the proposed framework against experimental data is an important step toward assessing its applicability to realistic phototaxis experiments.

\bibliographystyle{plain}
\bibliography{refs} 

\end{document}